\documentclass[suppldata]{interact}

\usepackage[numbers,sort&compress]{natbib}
\bibpunct[, ]{[}{]}{,}{n}{,}{,}

\makeatletter
\def\NAT@def@citea{\def\@citea{\NAT@separator}}
\makeatother

\usepackage[OT2,OT1]{fontenc}

\DeclareFontFamily{OT2}{cmr}{\hyphenchar\font45 }
\DeclareFontShape{OT2}{cmr}{m}{n}{<->wncyr10}{}
\DeclareFontShape{OT2}{cmr}{m}{it}{<->wncyi10}{}
\DeclareFontShape{OT2}{cmr}{m}{sc}{<->wncysc10}{}
\DeclareFontShape{OT2}{cmr}{b}{n}{<->wncyb10}{}
\DeclareFontShape{OT2}{cmr}{bx}{n}{<->ssub*wncyr/b/n}{}

\DeclareFontFamily{OT2}{cmss}{\hyphenchar\font45 }
\DeclareFontShape{OT2}{cmss}{m}{n}{<->wncyss10}{}

\DeclareRobustCommand\cyr{\fontencoding{OT2}\selectfont}
\DeclareTextFontCommand{\textcyr}{\cyr}

\DeclareUnicodeCharacter{0306}{}

\usepackage{epstopdf}
\usepackage[caption=false]{subfig}

\usepackage[numbers,sort&compress]{natbib}
\bibpunct[, ]{[}{]}{,}{n}{,}{,}
\makeatletter
\def\NAT@def@citea{\def\@citea{\NAT@separator}}
\makeatother

\theoremstyle{plain}
\newtheorem{theorem}{Theorem}[section]

\newtheorem{corr}[theorem]{Corollary}
\newtheorem{prop}[theorem]{Proposition}

\theoremstyle{definition}
\newtheorem{deff}[theorem]{Definition}

\theoremstyle{remark}
\newtheorem{comm}{Remark}

\newcommand{\R}{\mathbb R}

\newcommand{\C}{\mathbb C}

\DeclareMathOperator{\supp}{supp}

\newcommand{\p}{\partial}

\newcommand{\dbar}{\bar\partial}

\newcommand{\bc}{\mathbb{B}}
\newcommand{\holb}{Hol(D,\bc)}
\newcommand{\hpb}{H^p(D,\bc)}

\newcommand{\hw}{H_w(D, \bc)}
\newcommand{\hwp}{H^p_{w}(D,\bc)}
\newcommand{\sca}{\text{Sc}\,}
\newcommand{\vect}{\text{Vec}\,}

\usepackage{xcolor}

\begin{document}


\title{An Atomic Representation for Bicomplex Hardy Classes}

\author{
\name{William L. Blair\thanks{CONTACT William L. Blair. Email: wblair@uttyler.edu} }
\affil{Department of Mathematics, The University of Texas at Tyler, Tyler, TX 75799}
}

\maketitle

\begin{abstract}
We develop representations for bicomplex-valued functions in Hardy classes that generalize the complex holomorphic Hardy spaces. Using these representations, we show these functions have boundary values in the sense of distributions that are representable by an atomic decomposition, and we show continuity of the Hilbert transform on this class of distributional boundary values. 
\end{abstract}


\section{Introduction}

In this paper, we exploit relationships between complex-valued Hardy classes of functions defined on the complex unit disk and certain Hardy classes of bicomplex-valued functions that lead to useful representation formulas. 

In \cite{GHJH2}, G. Hoepfner and J. Hounie proved that holomorphic Hardy space functions have boundary values in the sense of distributions and those distributional boundary values are representable by an atomic decomposition. While this result is familiar to experts in harmonic analysis for certain classes of distributions on the boundary of the upper-half space, Hoepfner and Hounie provided a proof for the case of holomorphic functions on the disk, for which there was not an extent proof. Note, the conditions for a holomorphic function to have a boundary value in the sense of distributions were previously considered by E. Straube in \cite{Straube}, the atomic decomposition result for Hardy spaces of the disk was previously considered for the $p=1$ case by P. Koosis in \cite{Koosis}, and the atomic decomposition of disk Hardy spaces is described by R. Coifman in \cite{Coif2}. In \cite{WB}, the author showed that functions which solve nonhomogeneous Cauchy-Riemann equations and satisfy the same size condition as the holomorphic Hardy spaces are representable as the sum of a holomorphic Hardy space function and an error term constructed with the area integral from the classic Cauchy-Pompieu formula. Using this pointwise representation, the author showed these generalized Hardy space functions have boundary values in the sense of distributions and these distributional boundary values are representable as the sum of an atomic decomposition and a well-controlled error term. This extends the result from \cite{GHJH2} to classes of nonholomorphic functions. Also in \cite{WB}, the author extends continuity of the Hilbert transform to these classes of distributions by appealing to the associated result for holomorphic functions in \cite{GHJH2}. This provides another class of distributions associated with a small $p$-class of functions where the Hilbert transform is a continuous operator. Both the pointwise and boundary representations are shown to hold for functions in Hardy classes that are solutions to higher-order nonhomogeneous Cauchy-Riemann equations where the function and its derivatives have finite Hardy space norm. The boundary representation is used to show that the Hilbert transform is a continuous operator on the class of distributions associated with these Hardy classes also. Following the previously described success, we return to the question posed in the introduction of \cite{WB} which is: ``What classes of functions, other than the holomorphic Hardy space functions, have boundary values in the sense of distributions on the circle that have an atomic decomposition?'' 

We consider Hardy classes of functions defined on the complex unit disk that take values in the bicomplex numbers. The bicomplex numbers are a higher-dimensional extension of the complex numbers that, unlike quaternions, are commutative with respect to multiplication. The analysis of functions of bicomplex variables and bicomplex-valued functions is an active area of research as it provides a setting for a generalization of complex analysis different than the well-studied Clifford analysis and analysis of functions of several complex variables. For background on bicomplex numbers and bicomplex-valued functions as well as some recent examples of research in analysis that involve the bicomplex numbers, see \cite{BCTransmutation, BCBergman, FundBicomplex, BicomplexHilbert, KravAPFT, BCHolo,ComplexSchr}. The Hardy classes that we consider will be functions that satisfy a bicomplex generalization of the Cauchy-Riemann equation and satisfy a size condition analogous to that of the complex-valued Hardy space functions. This consideration will include both a bicomplex analogue of holomorphic functions as well as nonholomorphic functions. We prove that functions in these bicomplex Hardy classes are intimately connected with functions in the holomorphic Hardy spaces of complex-valued functions as well as the generalized Hardy classes of solutions to nonhomogeneous Cauchy-Riemann equations considered in \cite{WB} by showing every function in the bicomplex Hardy classes is representable as a linear combination of a holomorphic Hardy space function and a generalized Hardy class function. This is a nontrivial application of the generalized Hardy classes considered in \cite{WB}. Using this representation, we find these bicomplex-valued functions inherit the familiar boundary behavior of existence of an $L^p$ boundary value and convergence to these $L^p$ boundary values in the $L^p$ norm of the classic Hardy spaces. Also, we show that these functions have boundary values in the sense of distributions, and by appealing to a combination of the atomic decomposition results in \cite{GHJH2} and \cite{WB}, we show these distributional boundary values can be represented by a sum of atomic decompositions. This atomic decomposition representation directly leads to showing the Hilbert transform is a continuous operator on this class of distributions. As in \cite{WB}, all of these results are shown to extend to functions in Hardy classes of solutions to higher-order generalizations of the bicomplex Cauchy-Riemann equation.

We outline this paper. In Section \ref{background}, we provide background concerning the classic holomorphic Hardy spaces, the generalized Hardy spaces of complex solutions to nonhomogeneous Cauchy-Riemann equations, and bicomplex numbers. Previously unconsidered extensions of certain complex results will be included here as well as generalizations of certain tools to the bicomplex setting. In Sections \ref{DefsSect} and \ref{bcholohardysection}, we define the notion of holomorphicity that we consider, the corresponding bicomplex holomorphic Hardy spaces, and show these functions have representations in terms of functions in the complex holomorphic Hardy spaces. In Section \ref{generalizedbchardysection}, we define the bicomplex generalized Hardy spaces and show they are intimately connected to the previously studied complex generalized Hardy spaces. In Section \ref{atomicdecompsection}, the bicomplex Hardy classes are shown to have boundary values in the sense of distributions that are representable by a sum of atomic decompositions of boundary distributions of complex Hardy class functions. In Section \ref{higherordersection}, natural analogues of the work in the previous sections will be shown for Hardy classes of solutions to higher-order differential equations in the bicomplex setting.

\section{Background}\label{background}

We begin this section with a brief description of notation. We work exclusively on the domain $D$ which we take to be the disk of radius one centered at the origin in the complex plane. By $L^p(S)$, we denote the complex-valued functions defined on a set $S$ with integrable modulus raised to the $p^{\text{th}}$-power. We represent the space of distributions on the boundary of the unit disk $\p D$ by $\mathcal{D}'(\p D)$. We represent the space of $k$-times continuously differentiable functions on a set $S$ by $C^k(S)$. Finally, in an effort to disambiguate between other notions of conjugation that will arise below, we denote the complex conjugate of $z \in \C$ by $z^*$, i.e., if $z = x + iy$, where $x, y \in \R$, then $z^* = x - iy$. 

\subsection{Holomorphic Hardy Spaces}\label{holohardy}

We recall definitions and facts about the classic holomorphic Hardy spaces that we generalize. 

\begin{deff}
    We define $Hol(D)$ to be the space of functions $f: D \to \C$ such that 
    \[
        \frac{\p f}{\p z^*} = 0. 
    \]
\end{deff}

\begin{deff}
    For $0 < p < \infty$, we define $H^p(D)$ to be the space of functions $f \in Hol(D)$ such that 
    \[
        ||f||_{H^p} := \sup_{0 < r < 1} \left( \int_0^{2\pi} |f(re^{i\theta})|^p \,d\theta \right)^{1/p} < \infty.
    \]
\end{deff}

\begin{theorem}[\cite{Duren}]\label{bvcon}
A function $f \in H^p(D)$, $0 < p < \infty$, has nontangential boundary values $f_{nt} \in L^p(\partial D)$ at almost every point of $\p D$, 
\[
\lim_{r\nearrow 1} \int_0^{2\pi} |f(re^{i\theta})|^p \, d\theta = \int_0^{2\pi} |f_{nt}(e^{i\theta})|^p \,d\theta,
\]
and
\[
\lim_{r\nearrow 1} \int_0^{2\pi} |f(re^{i\theta})- f_{nt}(e^{i\theta})|^p \, d\theta = 0.
\]
\end{theorem}

Since the Lebesgue spaces for $0 < p < 1$ are not as useful as their $p \geq 1$ counterparts, we consider a more general kind of boundary value for functions in $H^p(D)$. 

\begin{deff}\label{bvcircle}Let $f: D \to \C$. We say that $f$ has a boundary value in the sense of distributions, denoted by $f_b \in \mathcal{D}'(\p D)$, if, for every $\varphi \in C^\infty(\partial D)$, the limit
            \[
             \langle f_b, \varphi \rangle := \lim_{r \nearrow 1} \int_0^{2\pi} f(re^{i\theta}) \, \varphi(e^{i\theta}) \,d\theta
            \]
            exists.
            
\end{deff}

In \cite{GHJH2}, it is shown that functions in $H^p(D)$ have distributional boundary values. 

\begin{theorem}[Theorem 3.1 \cite{GHJH2}]\label{GHJH23point1}
For $f \in Hol(D)$, the following are equivalent:
\begin{enumerate}
    \item For every $\phi \in C^\infty(\p D)$, there exists the limit
    \[
             \langle f_b, \phi \rangle := \lim_{r \nearrow 1} \int_0^{2\pi} f(re^{i\theta}) \, \phi(e^{i\theta}) \,d\theta.
            \]

    \item There is a distribution $f_b \in \mathcal{D}'(\p D)$ such that $f$ is the Poisson integral of $f_b$
    \[
             f(re^{i\theta}) = \frac{1}{2\pi}\langle f_b, P_r(\theta - \cdot) \rangle,
            \]
    where 
    \[
        P_r(\theta) = \frac{1-r^2}{1-2r\cos(\theta) +r^2}
    \]
    is the Poisson kernel on $D$. 

    \item There are constants $C>0$, $\alpha \geq 0$, such that 
    \[
        |f(re^{i\theta})| \leq \frac{C}{(1-r)^\alpha},
    \]
    for $0 \leq r < 1$.
\end{enumerate}
\end{theorem}

\begin{theorem}[Corollary 3.1 \cite{GHJH2}]\label{GHJH23point1corr}
The functions in $H^p(D)$, $0 < p < \infty$, satisfy (3) in Theorem \ref{GHJH23point1}
\end{theorem}

Furthermore, not only do holomorphic Hardy spaces functions have boundary values in the sense of distributions, but those distributional boundary values are representable by an atomic decomposition. 

\begin{deff}
    For $0 < p \leq 1$, we say a measurable function $a(e^{i\theta})$ on $\p D$ is a $p$-atom when $a$ satisfies
    \begin{itemize}
    \item $\supp(a) \subset J$, where $J$ is an arc in $\p D$ (that could be all of $\p D$),
    \item $|a(\theta)|\leq |J|^{-1/p}$,
    \item $\int_0^{2\pi} a(\theta) \, \theta^k \,d\theta = 0,$ for $k \leq \frac{1}{p} -1$,
\end{itemize}
\end{deff}

\begin{theorem}[Theorem 2.2 \cite{GHJH2}]\label{GHJH2twopointtwo}
    For $0 < p \leq 1$ and $f \in H^p(D)$, $f$ has a distributional boundary value $f_b$ and there exist a sequence $\{a_j\}$ of $p$-atoms and a sequence $\{c_j\} \in \ell^p(\C)$ such that 
    \[
        f_b = \sum_{j} c_j a_j
    \]
    in the topology of $\mathcal{D}'(\p D)$. 
\end{theorem}

In \cite{GHJH2}, Theorem \ref{GHJH2twopointtwo} was used to show that the Hilbert transform on the circle is a continuous operator on the collection of distributional boundary values of functions in $H^p(D)$, for $0 < p \leq 1$. We define the Hilbert transform below and recall a classic result attributed to Marcel Riesz, which can be found, for example, in \cite{Rep}. 

\begin{deff}
        For any $f \in L^1(\p D)$, the Hilbert transform $H(f)$ is given by 
        \[
            H(f)(e^{i\theta}) := \lim_{\epsilon \to 0}\frac{1}{\pi} \int_{\epsilon \leq |t| \leq \pi} \frac{f(e^{i(\theta-t)})}{2\tan(t/2)}\,dt.
        \]
\end{deff}

\begin{theorem}[M. Riesz]\label{hilbertpgreaterthanone}
Let $u \in L^p(\p D)$, $1 < p < \infty$. The Hilbert transform of $u$, denoted by $H(u)$, is in $L^p(\p D)$, and 
\[
        ||H(u)||_{L^p(\p D)}\leq C ||u||_{L^p(\p D)},
\]
where $C$ is a constant that depends only on $p$.
\end{theorem}

In \cite{GHJH2}, the following definitions are used to realize the boundary values in the sense of distributions of holomorphic Hardy space functions as the class of objects that extends the Hilbert transform to the case of small $p$. The result proved in \cite{GHJH2} follows the definitions below.

\begin{deff}\label{atomicnorm}
    For $0< p \leq 1$ and $f_b \in \mathcal{D}'(\p D)$ with atomic decomposition $f_b = \sum_{n=1}^\infty c_n a_n$, where $\{c_n\} \in \ell^p(\C)$ and $\{a_n\}$ are a sequence of $p$-atoms, the atomic norm is defined by
    \[
        ||f_b||_{at} := \left( \inf \sum_{k=1}^\infty |c_k|^p\right)^{1/p},
    \]
    where the infimum is taken over the collection of all coefficient sequences of equivalent atomic decompositions. 
\end{deff}

\begin{deff}
    For $0 < p \leq 1$, we define $(H^p(D))_b$ to be the collection of boundary values in the sense of distributions of functions in $H^p(D)$.
\end{deff}

\begin{theorem}[Theorem 6.2, Corollary 6.2 \cite{GHJH2}]\label{hilbertholohardy}
The Hilbert transform is a continuous operator on $(H^p(D))_b$ with the atomic norm. 
\end{theorem}

Before we proceed to the necessary background material for the bicomplex numbers, we describe an integral operator and include associated results that we use in the complex setting. We define and use a generalization of this integral operator later in the bicomplex setting.

\begin{deff}\label{vekopondisk}
For $f: D \to \mathbb{C}$ and $z  \in \mathbb{C}$, we denote by $T(\cdot)$ the integral operator defined by 
\[
T(f)(z) = -\frac{1}{\pi} \iint_D \frac{f(\zeta)}{\zeta - z}\, d\xi\,d\eta,
\]
whenever the integral is defined.
\end{deff}

This $T(\cdot)$ comes directly from the classic Cauchy-Pompieu theorem. For $f \in L^1(D)$, 
\[
    \frac{\p}{\p z^*} T(f) = f,
\]
i.e., $T(\cdot)$ is a right-inverse operator to $\frac{\p}{\p z^*}$, see \cite{Vek, BegBook}. The next few results describe characteristics of $T(\cdot)$ and are used throughout this work. We include them here for completeness.

\begin{theorem}[Theorem 1.26 \cite{Vek} ] \label{Vekonepointtwentysix}
If $f \in L^q(D)$, $1 \leq q \leq 2$, then $T(f) \in L^\gamma(D)$, $1 < \gamma < \frac{2q}{2-q}$. 
\end{theorem}

\begin{theorem}[Theorem 1.4.7 \cite{KlimBook} ]\label{onefourseven}
If $f \in L^q(D)$, $1 < q \leq 2$, then $T(f)|_{\p D(0,r)} \in L^\gamma(\p D(0,r))$, $0 < r \leq 1$, where $\gamma$ satisfies $1 < \gamma < \frac{q}{2-q}$, and 
\[
||Tf||_{L^\gamma(\p D(0,r))} \leq C ||f||_{L^q(D)},
\]
where $C$ is a constant that does not depend on $r$ or $f$. 
\end{theorem}

\begin{theorem}[Theorem 1.4.8 \cite{KlimBook} ]\label{onefoureight}
If $f \in L^q(D)$, $1 < q \leq 2$, then 
\[
\lim_{r\nearrow 1} \int_0^{2\pi} |T(f)(e^{i\theta}) - T(f)(re^{i\theta})|^\gamma \,d\theta = 0, 
\]
for $1 < \gamma < \frac{q}{2-q}$. 
\end{theorem}

\begin{theorem}[Theorem 1.19 \cite{Vek} ]\label{onepointnineteen}
For $f \in L^q(D)$, $q > 2$, $T(f) \in C^{0,\alpha}(\overline{D})$, with $\alpha = \frac{ q-2}{q}$. 
\end{theorem}

\subsection{Bicomplex Numbers}

In this section (and those that follow), we use the notation and terminology from \cite{BCTransmutation, BCBergman} to describe the definitions and background concerning the bicomplex numbers that we use throughout. Additional references for background on the bicomplex numbers and functions taking bicomplex values include \cite{FundBicomplex, BicomplexHilbert, KravAPFT, BCHolo, ComplexSchr}, as found in \cite{BCTransmutation, BCBergman}.

The bicomplex numbers $\bc$, as a set, can be described as $\C^2$. Letting $\bc$ inherit the usual addition and multiplication by complex scalars, then with the multiplication defined by 
\[
    (z_1, z_2)(w_1,w_2) = (z_1w_1 - z_2w_2, z_1w_2 + z_2 w_1),
\]
for any $(z_1,z_2), (w_1,w_2) \in \C^2$, $\bc$ is a commutative algebra over $\C$. This multiplication is the same as identifying $(z_1, z_2) \in \C^2$ with 
\[
    z_1 + j z_2,
\]  
where $j^2 = -1$, and applying the usual rules for complex number multiplication. If we consider the elements 
\[
    p^{\pm} := \frac{1}{2}(1\pm ij),
\]
then we see that 
\[
p^+ p^+ = p^+,\quad p^- p^- = p^-, 
\]
and
\[
    p^+ p^- = 0.
\]
So, $p^\pm$ are idempotent, and $\bc$ has zero divisors.

\begin{deff}
    We define the scalar part of $z = z_1 + jz_2 \in \bc$ as 
    \[
        \sca z = z_1,
    \]
    and the vector part of $z = z_1 + jz_2 \in \bc$ as 
    \[
        \vect z = z_2.
    \]
\end{deff}

\begin{deff}
    We define the bicomplex conjugate of $z = z_1 + jz_2 \in \bc$ as 
    \[
        \overline{z} := z_1 - jz_2.
    \]
\end{deff}

\begin{comm}
    To differentiate between bicomplex conjugation and the usual complex conjugation, we will denote complex conjugation for $u \in \mathbb{C}$ by $u^*$, i.e., if $u = x + iy \in \C$, then $u^* = x - i y$. So, complex-valued holomorphic functions will be those $f: D \to \C$ such that 
    \[
        \frac{\p f}{\p z^*} = 0,
    \]
    complex-valued anti-holomorphic functions $f$ will satisfy 
    \[
        \frac{\p f}{\p z} = 0,
    \]
    and we will denote the respective spaces of these functions by $Hol(D)$ and $\overline{Hol(D)}$. 
\end{comm}

\begin{prop}[Proposition 1 \cite{BCTransmutation}]\label{everybchasplusandminus}
    Let $w \in \bc$. There exist unique $w^{\pm} \in \C$ such that 
    \[
        w = p^+ w^+ + p^- w^-.
    \]
    Furthermore, 
    \[
    w^{\pm} = \sca w \mp i \vect w.
    \]
\end{prop}

The next few definitions bring into the bicomplex setting familiar objects of analysis.

\begin{deff}
    For $w \in \bc$, we define the bicomplex norm of $w$, denoted by $||\cdot||_{\bc}$,  as
    \[
        ||w||_{\bc} := \sqrt{\frac{|w^+|^2 + |w^-|^2}{2}},
    \]
    where $|w^{\pm}|$ is the usual complex modulus. 
\end{deff}

\begin{comm}\label{basicestimate}
    Note that by the definition of $||\cdot||_{\bc}$ we have, for $w = p^+ w^+ + p^- w^- \in \bc$, 
    \[
        \frac{1}{\sqrt{2}} |w^\pm| \leq ||w||_{\bc} \leq \frac{1}{\sqrt{2}}\left( |w^+| + |w^-|\right), 
    \]
    and for $w,v \in \bc$, we have
    \[
        ||wv||_\bc \leq \sqrt{2} \, ||w||_\bc \, ||v||_\bc.
    \]
\end{comm}

\begin{deff}
    For $0 < p < \infty$, we define the bicomplex Lebesgue space $L^p(D, \bc)$ to be the set of functions $f: D \to \bc$ such that 
    \[  
        ||f||_{L^p(D,\bc)}:= \left( \iint_D ||f(z)||^p_{\bc} \,dx\,dy \right)^{1/p} < \infty.
    \]
    Similarly, we define $L^p(\p D, \bc)$ to be the set of functions $g: \p D \to \C$ such that 
    \[
        ||g||_{L^p(\p D,\bc)}:= \left( \int_{\p D} ||g(z)||^p_{\bc} \,d\sigma(z) \right)^{1/p} < \infty,
    \]
    where $\sigma$ is Lebesgue surface measure on $\p D$. 
\end{deff}

The next proposition associates a function's inclusion in the bicomplex Lebesgue spaces with the inclusion of the function's complex components from its idempotent representation in the associated $\mathbb{C}$-valued $L^p$ spaces. 

\begin{prop}\label{propLqiff}
    For $p$ a positive real number, a function $f = p^+ f^+ + p^- f^- \in L^p(D, \bc)$ if and only if $f^+, f^- \in L^p(D)$. The same result holds for $f \in L^p(\p D, \bc)$.
\end{prop}

\begin{proof}
    If $f \in L^p(D, \bc)$, then 
    \begin{align*}
        \iint_D |f^\pm(z)|^p \,dx\,dy \leq 2^{p/2} \iint_D ||f(z)||_{\bc}^p \,dx\,dy < \infty. 
    \end{align*}
    If $f^\pm \in L^p(D)$, then 
    \begin{align*}
        \iint_D ||f(z)||_{\bc}^p \,dx\,dy &\leq 2^{-p/2} \iint_D (|f^+(z)|+|f^-(z)|)^p\,dx\,dy \\
        &\leq C_p \left( \iint_D |f^+(z)|^p \,dx\,dy + \iint_D |f^-(z)|^p \,dx\,dy\right) < \infty,
    \end{align*}
    where $C_p$ is a constant that depends on only $p$. The proof for $f \in L^p(\p D,\bc)$ is the same. 

\end{proof}

Next, we define the differential operators that we consider. 

\begin{deff}
    We define the bicomplex differential operators $\p$ and $\dbar$ as 
    \[
        \p := \frac{1}{2} \left( \frac{\p}{\p x} - j \frac{\p }{\p y}\right)
    \]
    and
    \[
        \dbar := \frac{1}{2} \left( \frac{\p}{\p x} + j \frac{\p }{\p y}\right).
    \]
\end{deff}

\begin{comm}\label{minusfunctionisholo}
    Note that the definition of $\p$ and $\dbar$ above is equivalent to 
    \[
        \p:= p^+ \frac{\p}{\p z^*} + p^- \frac{\p }{\p z}
    \]
    and 
    \[
        \dbar := p^+ \frac{\p}{\p z} + p^- \frac{\p }{\p z^*},
    \]
    where $\frac{\p}{\p z}$ and $\frac{\p }{\p z^*}$ are the usual complex partial differential operators. Also, by considering the representation $w = p^+ w^+ + p^- w^-$, we observe that $\dbar \, w = 0$ if and only if 
    \[
        \frac{\p w^+}{\p z} = 0 = \frac{\p w^-}{\p z^*}, 
    \]
    i.e., $(w^+)^*,w^- \in Hol(D)$. 
\end{comm}

\begin{deff}
    We define the bicomplexification $\widehat{u}$ of a complex number $u = x + iy$ by 
    \[
        \widehat{u} = x + jy. 
    \]
\end{deff}

For any complex variable $u$, we have the following:
\begin{align*}
    \dbar \,\widehat{u} = 0,
\end{align*}
\begin{align*}
    \p \, \widehat{u}^n = n \widehat{u}^{n-1},
\end{align*}
\[
    \p \, \widehat{u^*} = 0,
\]
and 
\begin{align*}
    \dbar \,\widehat{u^*}^{n} = n \widehat{u^*}^{n-1}.
\end{align*}
So, with respect to $\dbar$, polynomials in the bicomplexification of the complex conjugate of a complex variable play the same role as polynomials in the complex conjugation of a complex variable with respect to $\frac{\p}{\p z^*}$. 

Next, we introduce a right-inverse operator to $\dbar$ in the complex setting. This operator was introduced in \cite{FundBicomplex} and generalizes the classic Cauchy-Pompieu operator from \cite{Vek} in the bicomplex setting. 

\begin{theorem}[\cite{Vek, FundBicomplex, BCTransmutation, BCBergman, conjbel}]\label{Thm: bctoperator}
    For $f \in L^p(D, \bc)$, $p \geq 1$, the bicomplex Theodorescu operator $T_{\bc}(\cdot)$ defined by 
    \[
        T_{\bc} f(z) := p^+\left(-\frac{1}{\pi} \iint_D \frac{f^+(\zeta)}{\zeta^* - z^*} \,d\eta\,d\xi \right) + p^-\left(- \frac{1}{\pi}\iint_{D} \frac{ f^-(\zeta)}{\zeta- z}\,d\eta\,d\xi \right) 
    \]
    exists and 
    \[
        \dbar \, T_{\bc} (f) = f. 
    \]
\end{theorem}

The next theorem generalizes Theorems \ref{Vekonepointtwentysix}, \ref{onefourseven}, \ref{onefoureight}, and \ref{onepointnineteen} from the complex setting to the setting of bicomplex-valued functions. 

\begin{theorem}\label{bctbehavior}
For every $f \in L^q(D, \bc)$, $1 \leq q \leq 2$, $T_\bc(f) \in L^\gamma(D,\bc)$, $1 < \gamma < \frac{2q}{2-q}$, $T_\bc(f)|_{\p D(0,r)} \in L^\gamma(\p D(0,r), \bc)$, $0 < r \leq 1$, where $\gamma$ satisfies $1 < \gamma < \frac{q}{2-q}$,  
\[
||T_\bc(f)||_{L^\gamma(\p D(0,r),\bc)} \leq C ||f||_{L^q(D,\bc)},
\]
where $C$ is a constant that does not depend on $r$ or $f$, and
\[
\lim_{r\nearrow 1} \int_0^{2\pi} ||T_\bc(f)(e^{i\theta}) - T_\bc(f)(re^{i\theta})||_\bc^\gamma \,d\theta = 0, 
\]
for $1 \leq \gamma < \frac{q}{2-q}$. For $f \in L^q(D,\bc)$, $q > 2$, $T(f) \in C^{0,\alpha}(\overline{D},\bc)$, with $\alpha = \frac{ q-2}{q}$. 
\end{theorem}

\begin{proof}

    Let $f \in L^q(D,\bc)$, $q\geq 1$. By the definitions of the integral operators $T(\cdot)$ and $T_\bc(\cdot)$, observe that
    \begin{align*}
    	 T_{\bc} f(z) :&= p^+\left(-\frac{1}{\pi} \iint_D \frac{f^+(\zeta)}{\zeta^* - z^*} \,d\eta\,d\xi \right) + p^-\left(- \frac{1}{\pi}\iint_{D} \frac{ f^-(\zeta)}{\zeta- z}\,d\eta\,d\xi \right) \\
	 &= p^+  ( T((f^+)^*))^*+ p^- T(f^-)(z).
    \end{align*}
    Now, since $f \in L^q(D,\bc)$ implies $f^\pm \in L^q(D)$ (and $(f^\pm)^* \in L^q(D)$), by Proposition \ref{propLqiff}, and $|(T((f^+)^*))^* |= |T((f^+)^*)|$, it follows, by Remark \ref{basicestimate} and Theorem \ref{Vekonepointtwentysix}, that we have
    \begin{align*}
    	\iint_{D} ||T_\bc(f)(z)||_{\bc}^\gamma\,dx\,dy
	&\leq C_\gamma\left( \iint_D |T((f^+)^*)(z) |^\gamma \,dx\,dy+ \iint_D |T(f^-)(z)|^\gamma \,dx\,dy\right) < \infty,
    \end{align*}
    for any $\gamma$ satisfying $1 < \gamma< \frac{2q}{2-q}$, where $C_\gamma$ is a constant that only depends on only $\gamma$. Hence, $f \in L^\gamma(D,\bc)$, for $1 < \gamma< \frac{2q}{2-q}$.  
    
    Appealing to Theorem \ref{onefourseven}, instead of Theorem \ref{Vekonepointtwentysix},  we have, for $0 < r \leq 1$, 
    \begin{align*}
    	\int_{0}^{2\pi} ||T_\bc(f)(re^{i\theta})||_{\bc}^\gamma\,d\theta
	&\leq C_\gamma\left( \int_{0}^{2\pi} |T((f^+)^*)(re^{i\theta}) |^\gamma \,d\theta+ \int_{0}^{2\pi} |T(f^-)(re^{i\theta})|^\gamma \,d\theta\right) \\
	&\leq \tilde{C_\gamma}\left( ||f^+ ||_{L^q(D)}^\gamma + ||f^- ||_{L^q(D)}^\gamma \right) < \infty,
    \end{align*}
    for $0 < \gamma< \frac{q}{2-q}$, where $C_\gamma$ and $\tilde{C_\gamma}$ are constants that do not depend on $r$ or $f$.   Thus, $T_\bc(f)|_{\p D(0,r)} \in L^\gamma(\p D(0,r), \bc)$, for each $0 < r\leq 1$ and $\gamma$ satisfying $1 < \gamma < \frac{q}{2-q}$. By Remark \ref{basicestimate} again, we have
    \begin{align*}
     |f^\pm(z)| \leq \sqrt{2}\, ||f(z)||_{\bc},
    \end{align*}
  so
  \[
  	||f^\pm ||_{L^q(D)}^\gamma \leq 2^{\gamma/2} ||f ||_{L^q(D,\bc)}^\gamma.
  \]
  Therefore, 
    \begin{align*}
    		||T_{\bc}(f)||_{L^\gamma(\p D(0,r),\bc)} 
		&\leq  C \left( ||f^+ ||_{L^q(D)}^\gamma + ||f^- ||_{L^q(D)}^\gamma \right)^{1/\gamma} \\
		&\leq \tilde{C} ||f||_{L^q(D,\bc)},
    \end{align*}
   where $C$ and $\tilde{C}$ are constants that do not depend on $r$ or $f$. 
   
   For any $0 < r < 1$ and $1 < \gamma < \frac{q}{2-q}$, observe that 
   \begin{align*}
   	&\int_0^{2\pi} ||T_\bc(f)(e^{i\theta}) - T_\bc(f)(re^{i\theta})||_{\bc}^\gamma \,d\theta \\
	&\leq M_\gamma \left( \int_0^{2\pi}  |(T((f^+)^*))^*(e^{i\theta}) - (T((f^+)^*))^*(re^{i\theta})|^\gamma  \,d\theta+ \int_0^{2\pi} |T(f^-)(e^{i\theta}) - T(f^-)(re^{i\theta}) |^\gamma \,d\theta\right)\\
	&= M_\gamma \left( \int_0^{2\pi}  |T((f^+)^*)(e^{i\theta}) - T((f^+)^*)(re^{i\theta})|^\gamma \, d\theta + \int_0^{2\pi} |T(f^-)(e^{i\theta}) - T(f^-)(re^{i\theta}) |^\gamma \,d\theta \right),
   \end{align*}
   where $M_\gamma$ is a constant that depends on $\gamma$. By Theorem \ref{onefoureight}, we have
   \begin{align*}
   	&\lim_{r\nearrow 1} \int_0^{2\pi} ||T_\bc(f)(e^{i\theta}) - T_\bc(f)(re^{i\theta})||_{\bc}^\gamma \,d\theta\\
	&\leq \lim_{r\nearrow 1} M_\gamma \left( \int_0^{2\pi}  |T((f^+)^*)(e^{i\theta}) - T((f^+)^*)(re^{i\theta})|^\gamma \, d\theta + \int_0^{2\pi} |T(f^-)(e^{i\theta}) - T(f^-)(re^{i\theta}) |^\gamma \,d\theta \right) \\
	&= M_\gamma \left( \lim_{r\nearrow 1} \int_0^{2\pi}  |T((f^+)^*)(e^{i\theta}) - T((f^+)^*)(re^{i\theta})|^\gamma \, d\theta + \lim_{r\nearrow 1} \int_0^{2\pi} |T(f^-)(e^{i\theta}) - T(f^-)(re^{i\theta}) |^\gamma \,d\theta \right) \\
	&=0.
   \end{align*}
   Thus, 
   \[
   	\lim_{r\nearrow 1} \int_0^{2\pi} ||T_\bc(f)(e^{i\theta}) - T_\bc(f)(re^{i\theta})||_{\bc}^\gamma \,d\theta = 0.
   \]

   Finally, by Theorem \ref{onepointnineteen}, $T((f^+)^*)$ and $T(f^-)$ are in $C^{0,\alpha}(\overline{D})$, with $\alpha = \frac{q-2}{q}$, when $q>2$. Since $T((f^+)^*),T(f^-)\in C^{0,\alpha}(\overline{D})$, there exist $C_+$ and $C_-$ such that 
   \[
   	|  T((f^+)^*)(z) - T((f^+)^*)(w)| \leq C_+ |z- w|^\alpha
   \]
   and 
   \[
   		|T(f^-)(z) - T(f^-)(w) | \leq C_-|z-w|^\alpha,
   \]
 for all $z , w \in \overline{D}$. Since $T_\bc(f) = p^+  ( T((f^+)^*))^*+ p^- T(f^-)$, it follows that, for any $z,w \in \overline{D}$, we have
   \begin{align*}
        & ||T_\bc(f)(z)- T_\bc(f)(w)||_\bc \\
   	&\leq 2^{-1/2}\left( | ( T((f^+)^*))^*(z) - ( T((f^+)^*))^*(w)| + |T(f^-)(z) - T(f^-)(w) | \right)\\
	&= 2^{-1/2}\left( |  T((f^+)^*)(z) - T((f^+)^*)(w)| + |T(f^-)(z) - T(f^-)(w) | \right)\\
	&\leq 2^{-1/2}\left(C_+ |z- w|^\alpha + C_-|z-w|^\alpha \right) = C|z-w|^\alpha,
   \end{align*}
      where $C = \max\{2^{-1/2}C_+, 2^{-1/2}C_-\}$, and we have $T_\bc(f) \in C^{0,\alpha}(\overline{D},\bc)$, with $\alpha = \frac{q-2}{q}$.

\end{proof}

We also consider distributional boundary values in the context of bicomplex-valued functions. 

\begin{deff}\label{bcbvcircle}Let $f: D \to \bc$. We say that $f$ has a boundary value in the sense of distributions, denoted by $f_b \in \mathcal{D}'(\p D)$, if, for every $\varphi \in C^\infty(\partial D)$, the limit
            \[
             \langle f_b, \varphi \rangle := \lim_{r \nearrow 1} \int_0^{2\pi} f(re^{i\theta}) \, \varphi(e^{i\theta}) \,d\theta
            \]
            exists.
            
\end{deff}

With this definition, we extend Theorem 3.6 from \cite{WB3} to the bicomplex setting. This is a useful result concerning integrable functions and distributional boundary values. 

\begin{theorem}\label{lonedistbv}
For every $f \in L^1(D, \bc)$ such that $f$ has a $L^1(\p D,\bc)$ boundary value and 
\[
    \lim_{r \nearrow 1} \int_0^{2\pi} ||f(re^{i\theta}) - f(e^{i\theta})||_\bc \,d\theta = 0,
\]
$f_b$ exists and $f_b = f|_{\p D}$ as distributions. 
\end{theorem}

\begin{proof}
By Proposition \ref{everybchasplusandminus}, $f = p^+ f^+ + p^- f^-$, and by Proposition \ref{propLqiff}, $f^+, f^- \in L^1(D)$. Since $f$ has an $L^1(\p D, \mathbb{B})$ boundary value and $f = p^+ f^+ + p^- f^-$, it follows, by Proposition \ref{propLqiff}, that $f^+, f^- \in L^1(\p D)$. Now, since 
\[
    f(re^{i\theta}) - f(e^{i\theta}) = p^+ (f^+(re^{i\theta}) - f^+(e^{i\theta})) + p^- (f^-(re^{i\theta}) - f^-(e^{i\theta})),
\]
for every $r^{i\theta} \in D$, it follows that 
\begin{align*}
\int_0^{2\pi} |f^\pm (re^{i\theta}) - f^\pm(e^{i\theta})|\,d\theta
&\leq \sqrt{2} \int_0^{2\pi} ||f(re^{i\theta} - f(e^{i\theta})||_{\bc}\,d\theta ,
\end{align*}
for every $r \in (0,1)$. Thus, 
\[
    \lim_{r \nearrow 1}\int_0^{2\pi} |f^\pm (re^{i\theta}) - f^\pm(e^{i\theta})|\,d\theta
\leq \sqrt{2} \lim_{r \nearrow 1}\int_0^{2\pi} ||f(re^{i\theta} - f(e^{i\theta})||_{\bc}\,d\theta  =0.
\]
In other words, 
\[
\lim_{r \nearrow 1}\int_0^{2\pi} |f^\pm (re^{i\theta}) - f^\pm(e^{i\theta})|\,d\theta = 0.
\]
By Theorem 3.6 of \cite{WB3}, every $L^1(D)$ function with $L^1(\p D)$ boundary value that converges to that boundary value in the $L^1$ norm has a distributional boundary value and the distributional boundary value is equal to the $L^1(\p D)$ boundary value as distributions. Hence, $f^+_b$ and $f^-_b$ exist and $f^+_b=f^+|_{\p D}$ and $f^-_b=f^-|_{\p D}$. So, for $\varphi \in C^\infty(\p D)$, 
\begin{align*}
    \lim_{r\nearrow 1} \int_0^{2\pi} f(re^{i\theta}) \varphi(e^{i\theta}) \,d\theta 
    &= p^+\lim_{r\nearrow 1} \int_0^{2\pi} f^+(re^{i\theta}) \varphi(e^{i\theta}) \,d\theta + p^- \lim_{r\nearrow 1} \int_0^{2\pi} f(re^{i\theta}) \varphi(e^{i\theta}) \,d\theta < \infty,
\end{align*}
and $f_b$ exists. Since $f = p^+ f^+ + p^- f^-$, it follows that $f_b = p^+ f^+_b + p^- f^-_b$. Since $f^+_b=f^+|_{\p D}$ and $f^-_b=f^-|_{\p D}$, it follows that 
\begin{align*}
    f|_{\p D} &= p^+ f^+|_{\p D} + p^- f^-|_{\p D} \\
    &= p^+ f^+_b + p^- f^-_b\\
    &= f_b.
\end{align*}
\end{proof}

\subsection{Generalizations of the Hardy Spaces}

Recently, there has been progress towards building a more general Hardy space theory of complex-valued functions on the disk that is not dependent on holomorphicity. In these generalized Hardy classes, functions (and in the higher-order variants, their derivatives) are required to have a finite $H^p$ norm. However, instead of requiring holomorphicity of the function, the requirement for membership is that the function solves a certain generalization of the classic Cauchy-Riemann equation. For reference, some generalizations that have been considered include (but are certainly not limited to): 
\begin{itemize}
        \item Generalized analytic functions, i.e., solutions of the Vekua equation
        \[
            \frac{\p w}{\p z^*} = Aw + Bw^*
        \]
        in \cite{KlimBook}, as well as \cite{ PozHardy, CompOp, moreVekHardy, conjbel}. For a three-dimensional consideration, see \cite{threedvek}.
        \item Quasiconformal maps, i.e., solutions of the Beltrami equation
        \[
            \frac{\p w}{\p z^*} = \mu \frac{\p w}{\p z},
        \]
        where $\mu$ is a measurable function bounded in the $L^\infty$ norm by a constant less than one, see \cite{KlimBook, quasiHardy, quasiHardy2, quasiHardy3}.
        \item Solutions of the conjugate Beltrami equation
        \[
            \frac{\p w}{\p z^*} = \mu \frac{\p w^*}{\p z^*},
        \]
        where $\mu$ is a measurable function bounded in the $L^\infty$ norm by a constant less than one, see \cite{PozHardy, CompOp, moreVekHardy, conjbel}.
        \item Poly-analytic functions, i.e., solutions of 
        \[
            \left(\frac{\p}{\p z^*}\right)^n w = 0
        \]
        studied in \cite{polyhardy}.
        \item Solutions of the higher-order iterated Vekua equations
        \[
            \left(\frac{\p}{\p z^*} - A - B C(\cdot)\right)^n w = 0,
        \]
         studied in \cite{metahardy, WB3, WBD} for various classes of coefficients $A,B$.
         \item Solutions of the general nonhomogeneous Cauchy-Riemann equations of first-order
         \[
            \frac{\p w}{\p z^*} = f
         \]
         and higher order
         \[
            \left(\frac{\p }{\p z^*}\right)^n w = f
         \]
         found in \cite{WB}.
\end{itemize}

The main goals of this paper is to extend results found in \cite{WB} concerning representation and boundary behavior to the setting of Hardy-type classes of functions defined on $D$ that take values in $\bc$. For completeness, we recall the relevant results from those papers and provide extensions when they will be needed later. 

\begin{deff}\label{cgenhardydef}
For $0 < p < \infty$ and $f: D \to \C$, we define the Hardy classes $H^p_f(D)$ to be the collection of functions $w: D \to \C$ such that 
\[
    \frac{\p w}{\p z^*} = f
\]
and 
\[
    \sup_{0 < r < 1} \int_0^{2\pi} |w(re^{i\theta})|^p \,d\theta < \infty.
\]
\end{deff}

\begin{comm}
    For $f \equiv 0$, $H^p_f(D) = H^p_0(D)$ is exactly the classic holomorphic Hardy space $H^p(D)$. 
\end{comm}

Using the $T(\cdot)$ operator defined in Section \ref{holohardy}, we have the following representation for functions that satisfy nonhomogeneous Cauchy-Riemann equations.

\begin{theorem}[Theorem 1.16 \cite{Vek}]\label{oneonesix}
If $w: D \to \C$ satisfies $\frac{\p w}{\p z^*} = f \in L^1(D)$, then 
\begin{equation}\label{cp}
w(z) = \varphi(z) - \frac{1}{\pi} \iint_D \frac{f(\zeta)}{\zeta - z}\,d\eta\,d\xi,
\end{equation}
where $\zeta = \eta + i\xi$ and $\varphi \in Hol(D)$. 
\end{theorem}

The next result is clear by a direct computation. 

\begin{prop}
    For every $0 < p < \infty$, $f \in L^q(D)$, $q>2$, and $\varphi \in H^p(D)$, the function 
    \[
        w = \varphi + T(f)
    \]  
    is an element of $H^p_f(D)$. The result holds for $1 < q \leq 2$, so long as $p < \frac{q}{2-q}$.
\end{prop}

The next theorem is an extension of Theorem 2.10 from \cite{WB} where only the cases of $p \leq 1$ were considered. The argument is the same as in that paper, but we include it here for completeness and to have occasion to note that the restriction on the values of $p$ from the original result was unnecessary for existence of the point-wise representation. 

\begin{theorem}\label{thm: 2.10WBextension}
Let $0 < p < \infty$, $q > 2$, and $f \in L^q(D)$. Every $w \in H^p_f(D)$ has a representation 
\[
    w = \varphi + T(f)
\]
and $\varphi \in H^p(D)$. Also, the boundary value in the sense of distributions $w_b$ exists and, for $p\leq 1$, can be represented as 
\[
    w_b = \sum_{n=1}^\infty c_n a_n + T(f)_b,
\]
where $\{c_n\} \in \ell^p(\mathbb{C})$, $\{a_n\}$ is a collection of $p$-atoms, and $T(f)_b \in C^{0,\alpha}(\p D)$, $\alpha = \frac{q-2}{q}$. The result holds when $1 < q \leq 2$, so long as $p$ satisfies $p < \frac{q}{2-q}$, with $T(f)_b \in L^\gamma(\p D)$, for $1 < \gamma< \frac{q}{q-2}$.
\end{theorem}

\begin{proof}
    By assumption, $\frac{\p w}{\p z^*} = f \in L^q(D) \subset L^1(D)$, for $q > 1$. So, by Theorem \ref{oneonesix}, $w$ has a representation 
\begin{equation}\label{repofsecondkind}
w(z) = \varphi(z) + T(f)(z),
\end{equation}
where $\varphi$ is holomorphic in $D$. Observe that
\begin{align*}
&\int_0^{2\pi} |\varphi(re^{i\theta})|^p \,d\theta  \\
&= \int_0^{2\pi} |w(re^{i\theta}) - T(f)(re^{i\theta})|^p \,d\theta  \\
&\leq C_p \left(\int_0^{2\pi} |w(re^{i\theta})|^p \,d\theta + \int_0^{2\pi} | T(f)(re^{i\theta})|^p \,d\theta\right)\\
&\leq C_p\left(||w||^p_{H^p(D)} + \int_0^{2\pi} | T(f)(re^{i\theta})|^p \,d\theta \right),
\end{align*}
where $C_p$ is a constant that depends on only on $p$.

If $q>2$, then, by Theorem \ref{onepointnineteen}, $T(f) \in C^{0, \alpha}(\overline{D})$. Hence, there exists a constant $M>0$ such that $|T(f)(z)| \leq M$, for all $z \in \overline{D}$. So, for each $r \in (0,1)$,
\begin{align*}
    \int_0^{2\pi} | T(f)(re^{i\theta})|^p \,d\theta &\leq \int_0^{2\pi} M^p \,d\theta = 2\pi M^p  < \infty.
\end{align*}
Thus, 
\[
\sup_{0 < r < 1} \int_0^{2\pi} | T(f)(re^{i\theta})|^p \,d\theta < \infty.
\]
Therefore, 
\begin{align*}
\sup_{0 < r < 1} \int_0^{2\pi} |\varphi(re^{i\theta})|^p \,d\theta &\leq C_p\left(||w||^p_{H^p(D)} + \sup_{0 < r < 1}\int_0^{2\pi} | T(f)(re^{i\theta})|^p \,d\theta \right) \\
&\leq C_p\left(||w||^p_{H^p(D)} + 2\pi M^p \right) < \infty,
\end{align*}
and $\varphi \in H^p(D)$. 

Now, suppose that $1 < q \leq 2$ and $p$ satisfies $p < \frac{q}{2-q}$. For each $r$, let $D_{r,*} = \{\theta \in [0,2\pi] : |T(f)(re^{i\theta})| \leq 1 \}$ and $D^{r,*} = \{\theta \in [0,2\pi] : |T(f)(re^{i\theta})| > 1 \}$. By Theorem \ref{onefourseven}, since $f \in L^q(D), 1 < q \leq 2$, it follows that 
\[
\sup_{0 < r < 1}||T(f)||_{L^\gamma(\p D(0,r))} \leq C ||f||_{L^q(D)},
\]
where $\gamma$ satisfies $1 < \gamma < \frac{q}{2-q}$ and $C$ is a constant that does not depend on $f$. Since $p < \frac{q}{2-q}$, it follows that there exist a $\gamma$ such that $p < \gamma < \frac{q}{2-q}$, and 
\begin{align*}
&\int_0^{2\pi} | T(f)(re^{i\theta})|^p \,d\theta \\
&= \int_{D_{r,*}} | T(f)(re^{i\theta})|^p \,d\theta + \int_{D^{r,*}} | T(f)(re^{i\theta})|^p \,d\theta \\
&\leq 2\pi + \int_0^{2\pi} | T(f)(re^{i\theta})|^\gamma \,d\theta \\
&\leq 2\pi + C||f||_{L^q(D)} < \infty.
\end{align*}
Hence, 
\begin{align*}
\int_0^{2\pi} |\varphi(re^{i\theta})|^p \,d\theta 
& \leq  ||w||^p_{H^p(D)} + \int_0^{2\pi} | T(f)(re^{i\theta})|^p \,d\theta \\
&\leq  ||w||^p_{H^p(D)} +  2\pi + C||f||_{L^q(D)} < \infty,
\end{align*}
where the right hand side has no dependence on $r$. Thus, 
\[
\sup_{0 < r < 1} \int_0^{2\pi} |\varphi(re^{i\theta})|^p \,d\theta \leq ||w||^p_{H^p(D)} +  2\pi + C||f||_{L^q(D)} < \infty,
\]
and $\varphi \in H^p(D)$. 

Now, we restrict to $p\leq 1$. By Theorem \ref{GHJH2twopointtwo}, since $\varphi \in H^p(D)$, it follows that there exists $\{c_n\} \in \ell^p(\mathbb{C})$  and a sequence of $p$-atoms $\{a_n\}$ such that $\varphi_b = \sum_{n=1}^\infty c_n a_n$. By Theorems \ref{Vekonepointtwentysix}, \ref{onefourseven}, and \ref{onefoureight}, $T(f) \in L^1(D)$, $T(f)$ has an $L^1(\p D)$ boundary value $T(f)|_{\p D}$, and $T(f)$ converges to $T(f)|_{\p D}$ in the $L^1$ norm, for $q>1$. By Theorem 3.6 from \cite{WB3}, $T(f)_b$ exists and $T(f)|_{\p D} = T(f)_b$. Therefore, 
\[
    w_b = \sum_{n=1}^\infty c_n a_n + T(f)_b.
\]
By Theorem \ref{onepointnineteen}, if $q>2$, then $T(f)_b \in C^{0,\alpha}(\p D)$, $\alpha = \frac{q-2}{q}$. 

\end{proof}

\begin{corr}\label{genhardyrepcorr}
    For $0 < p < \infty$, $q > 2$, and $f \in L^q(D)$, $w = \varphi + T(f) \in H^p_f(D)$ if and only if $\varphi \in H^p(D)$. The result holds when $1 < q \leq 2$, so long as $p$ satisfies $p < \frac{q}{2-q}$.
\end{corr}

The next theorem describes the $L^p$ function boundary behavior of the elements of the classes $H^p_f(D)$. This was not addressed in \cite{WB} which focused on describing an atomic decomposition for distributional boundary values of the small $p$ classes. 

\begin{theorem}\label{thm: nonhomogHpbvcon}
For $0 < p < \infty$ and $q>2$, every $w \in H^p_f(D)$, where $f \in L^q(D)$, has a nontangential boundary value $w_{nt} \in L^p(\p D)$ and 
\[
\lim_{r \nearrow 1} \int_0^{2\pi} |w_{nt}(e^{i\theta}) - w(re^{i\theta}) |^p \, d\theta  = 0.
\]
The result holds when $1 < q \leq 2$, so long as $p$ satisfies $p < \frac{q}{2-q}$. 
\end{theorem}

\begin{proof}
    By Theorem \ref{thm: 2.10WBextension}, 
    \[
        w = \varphi + T(f)
    \]
    and $\varphi \in H^p(D)$. By Theorem \ref{bvcon}, $\varphi_{nt} \in L^p(\p D)$ and 
    \[
        \lim_{r \nearrow 1} \int_0^{2\pi} |\varphi_{nt}(e^{i\theta}) - \varphi(re^{i\theta}) |^p \, d\theta  = 0.
    \]
    By Theorems \ref{onefourseven} and \ref{onefoureight}, $T(f)_{nt}\in L^p(\p D)$ and 
    \[
        \lim_{r \nearrow 1} \int_0^{2\pi} |T(f)_{nt}(e^{i\theta}) - T(f)(re^{i\theta}) |^p \, d\theta  = 0.
    \]
    Thus, $w_{nt} = \varphi_{nt} + T(f)_{nt}$ exists and is in $L^p(\p D)$. For each $r \in (0,1)$, observe that 
    \begin{align*}
        &\int_0^{2\pi} |w_{nt}(e^{i\theta}) - w(re^{i\theta}) |^p \, d\theta\\
        &= \int_0^{2\pi} |\varphi_{nt}(e^{i\theta}) + T(f)_{nt}(e^{i\theta})- (\varphi(re^{i\theta})-T(f)(re^{i\theta})) |^p \, d\theta\\
        &\leq C_p\left(\int_0^{2\pi} |\varphi_{nt}(e^{i\theta}) - \varphi(re^{i\theta})|^p\,d\theta + \int_0^{2\pi} |T(f)_{nt}(e^{i\theta})-T(f)(re^{i\theta}) |^p \, d\theta \right),
    \end{align*}
    where $C_p$ is a constant that depends only on $p$. Since 
    \[
        \lim_{r \nearrow 1} \int_0^{2\pi} |\varphi_{nt}(e^{i\theta}) - \varphi(re^{i\theta}) |^p \, d\theta  = 0
    \]
    and
    \[
        \lim_{r \nearrow 1} \int_0^{2\pi} |T(f)_{nt}(e^{i\theta}) - T(f)(re^{i\theta}) |^p \, d\theta  = 0,
    \]
    it follows that 
    \[
        \lim_{r \nearrow 1} C_p\left(\int_0^{2\pi} |\varphi_{nt}(e^{i\theta}) - \varphi(re^{i\theta})|^p\,d\theta + \int_0^{2\pi} |T(f)_{nt}(e^{i\theta})-T(f)(re^{i\theta}) |^p \, d\theta \right) = 0.
    \]
    Thus, 
    \[
        \lim_{ r \nearrow 1} \int_0^{2\pi} |w_{nt}(e^{i\theta}) - w(re^{i\theta}) |^p \, d\theta = 0.
    \]
\end{proof}

In \cite{WB}, the author used the representation from Theorem \ref{thm: 2.10WBextension} (Theorem 2.10 in \cite{WB}) to extend Theorem \ref{hilbertholohardy} in the following way. 

\begin{deff}
For $p$ a positive real number and $w \in L^q(D)$, $q>1$, denote by $(H^p_w(D))_b$ the collection of distributional boundary values of the functions $f \in H^p_w(D)$. 
\end{deff}

\begin{deff}
We define the quasi-norm $||\cdot||_{(H^p_w(D))_b}$ on the class of distributions $(H^p_w(D))_b$ by 
\[
    ||f_b||_{(H^p_w(D))_b} := ||\sum_{n=1}^\infty c_n a_n ||_{at} + ||T(w)_b||_{L^\gamma(\p D)},
\]
where $f_b = \sum_{n=1}^\infty c_n a_n + T(w)_b$ is the distributional boundary value of $f \in H^p_w(D)$, $0 < p \leq 1$, $w \in L^q(D)$, $q>1$. 
\end{deff}

\begin{theorem}[Theorem 3.7 \cite{WB}]\label{hilbertcontnonhomoghardy}
For $0 < p \leq 1$ and $w \in L^q(D)$, $q>1$, the Hilbert transform is a continuous operator on $(H^p_w(D))_b$ with the quasi-norm $||\cdot||_{(H^p_w(D))_b}$.
\end{theorem}

\section{Definitions}\label{DefsSect}

We provide the relevant definitions for the function classes that we study.

\begin{deff}
We define the $\bc$-holomorphic functions on $D$ to be the functions $f: D \to \bc$ such that 
    \[
        \dbar f = 0.
    \]
    We denote by $\holb$ the space of all $\bc$-holomorphic functions on $D$.  Similarly, we define the $\bc$-anti-holomorphic functions on $D$ to be the functions $f: D \to \bc$ such that 
    \[
        \p f = 0.
    \]
    We denote by $\overline{\holb}$ the space of all $\bc$-anti-holomorphic functions on $D$.
\end{deff}

\begin{deff}
Let $w: D \to \bc$. We define $\hw$ to be those functions $f: D \to \bc$ such that 
\[
    \dbar f = w.
\]
\end{deff}

\begin{deff}
    For $0 < p < \infty$ and a function $f: D \to \bc$, we define the bicomplex-$H^p$ norm to be
    \[
        ||f||_{H^p_{\bc}} := \sup_{0 < r < 1}\left( \int_0^{2\pi} ||f(re^{i\theta})||^p_{\bc} \,d\theta\right)^{1/p}.
    \]
\end{deff}

\begin{comm}
    The bicomplex-$H^p$ norm is only truly a norm when $p \geq 1$ and is a quasi-norm for $0 < p < 1$. Following the convention in the literature for analogous objects, we call this quantity a norm regardless of the value of $p$.  
\end{comm}

\begin{deff}
    Let $0 < p < \infty$. We define the bicomplex Hardy spaces $\hpb$ to be those functions $f \in \holb$ such that 
    $||f||_{H^p_{\bc}} < \infty$.
\end{deff}

\section{Bicomplex Hardy Spaces}\label{bcholohardysection}

\subsection{Representation}

In this section, we prove a representation result that presents a connection between bicomplex-valued Hardy spaces and associated $\mathbb{C}$-valued Hardy spaces. 

\begin{theorem}\label{thm: bchardyrep}
    For $0 < p < \infty$, a function $f = p^+f^+ + p^-f^- \in \hpb$ if and only if $(f^+)^*,f^- \in H^p(D)$. 
\end{theorem}

\begin{proof}
    By Proposition \ref{everybchasplusandminus} and Remark \ref{minusfunctionisholo}, for every $f: D \to \bc$ there exists $f^+:D \to \C$ and $f^-: D \to \C$ such that $f = p^+ f^+ + p^- f^-$ and 
    \[
        \dbar f= 0
    \]
    if and only if 
    \[
        \frac{\p f^+}{\p z} = 0 = \frac{\p f^-}{\p z^*}.
    \]

    Suppose $f \in \hpb$. Then $f = p^+ f^+ + p^- f^-$ where $(f^+)^*, f^- \in Hol(D)$. Since 
    \[
        \frac{1}{\sqrt{2}}|w^\pm|\leq ||w||_{\bc},
    \]
    for all $w \in \bc$, it follows that, for $r \in (0,1)$, we have
    \begin{align*}
        \int_0^{2\pi} |(f^+)^*(re^{i\theta})|^p \,d\theta &= \int_0^{2\pi} |f^+(re^{i\theta})|^p \,d\theta \\
        &\leq \int_0^{2\pi} (\sqrt{2}\,||f(re^{i\theta})||_{\bc})^p \,d\theta \\
        &= 2^{p/2}\int_0^{2\pi} ||f(re^{i\theta})||_{\bc}^p \,d\theta \\
        &\leq 2^{p/2} \sup_{0 < r < 1}\int_0^{2\pi} ||f(re^{i\theta})||_{\bc}^p \,d\theta < \infty.
    \end{align*}
    Hence, 
    \[
        \sup_{0 < r < 1}  \int_0^{2\pi} |(f^+)^*(re^{i\theta})|^p \,d\theta \leq 2^{p/2} \sup_{0 < r < 1}\int_0^{2\pi} ||f(re^{i\theta})||_{\bc}^p \,d\theta < \infty.
    \]
    Similarly, 
    \begin{align*}
    \int_0^{2\pi} |f^-(re^{i\theta})|^p \,d\theta 
        &\leq \int_0^{2\pi} (\sqrt{2}\,||f(re^{i\theta})||_{\bc})^p \,d\theta \\
        &= 2^{p/2}\int_0^{2\pi} ||f(re^{i\theta})||_{\bc}^p \,d\theta \\
        &\leq 2^{p/2} \sup_{0 < r < 1}\int_0^{2\pi} ||f(re^{i\theta})||_{\bc}^p \,d\theta < \infty.
    \end{align*}
    Hence, 
    \[
        \sup_{0 < r < 1}  \int_0^{2\pi} |f^-(re^{i\theta})|^p \,d\theta \leq 2^{p/2} \sup_{0 < r < 1}\int_0^{2\pi} ||f(re^{i\theta})||_{\bc}^p \,d\theta < \infty.
    \]
    Thus, $(f^+)^*,f^- \in H^p(D)$.

    Suppose $f^+:D\to\C$ and $f^-:D\to\C$ are two functions such that $(f^+)^*, f^- \in H^p(D)$ and consider the function $f = p^+ f^+ + p^- f^-$. Now, $f: D \to \bc$. Observe that 
    \begin{align*}
        \dbar f &= \left(p^+ \frac{\p}{\p z} + p^- \frac{\p }{\p z^*}\right)(p^+f^+ + p^- f^-)\\
        &= (p^+)^2 \frac{\p f^+}{\p z} + p^+p^- \frac{\p f^-}{\p z} + p^- p^+ \frac{\p f^+}{\p z^*} + (p^-)^2 \frac{\p f^-}{\p z^*} \\
        &= 0.
    \end{align*}
    So, $f \in Hol(D, \bc)$. Also, since
    \[
    ||w||_{\bc} \leq \frac{1}{\sqrt{2}}(|w^+| + |w^-|),
    \]
    for all $w \in \bc$, it follows, for $r \in (0,1)$, that 
    \begin{align*}
        \int_0^{2\pi} ||f(re^{i\theta})||_{\bc}^p \,d\theta 
        &\leq \int_0^{2\pi} \left( \frac{1}{\sqrt{2}}(|f^+(re^{i\theta})| + |f^-(re^{i\theta})|)\right)^p \,d\theta \\
        &\leq \frac{C_p}{2^{p/2}} \left[ \int_0^{2\pi} |f^+(re^{i\theta})|^p \,d\theta + \int_0^{2\pi} |f^-(re^{i\theta})|^p \,d\theta \right] \\
        &= \frac{C_p}{2^{p/2}} \left[ \int_0^{2\pi} |(f^+)^*(re^{i\theta})|^p \,d\theta + \int_0^{2\pi} |f^-(re^{i\theta})|^p \,d\theta \right]  < \infty,
    \end{align*}
    where $C_p$ is a constant that depends only on $p$. Therefore $f \in \hpb$.

\end{proof}

\subsection{Boundary Behavior}

Using the representation result from the last subsection, we show that functions in $H^p(D,\bc)$ exhibit boundary behavior similar to that of the $\mathbb{C}$-valued Hardy classes. The argument requires a direct appeal to the boundary behavior results for the generalized Hardy classes from \cite{WB}, see Definition \ref{cgenhardydef} and Theorems \ref{thm: 2.10WBextension} and \ref{thm: nonhomogHpbvcon}.

\begin{theorem}\label{bcholobvcon}
    For $0 < p < \infty$, every $f = p^+ f^+ + p^- f^-\in \hpb$ such that $f^+ \in H^p_w(D)$, where $w \in L^q(D)$, and $q >2$ or $1 < q \leq 2$ and $p$ satisfies $p < \frac{q}{2-q}$, has a nontangential limit $f_{nt} \in L^p(\p D, \bc)$, and 
    \[
        \lim_{r \nearrow 1} \int_0^{2\pi} ||f_{nt}(e^{i\theta}) - f(re^{i\theta}) ||^p_{\bc} \, d\theta  = 0. 
    \] 
\end{theorem}

\begin{proof}
Let $f \in \hpb$. By Theorem \ref{thm: bchardyrep}, $f = p^+f^+ + p^- f^-$, where $(f^+)^*,f^- \in H^p(D)$. Suppose also that $f^+ \in H^p_w(D)$, for some $w \in L^q(D)$, $q>2$ (or if $1 < q \leq 2$, then $p$ satisfies $p < \frac{q}{2-q}$). By Theorems \ref{bvcon} and Theorem \ref{thm: nonhomogHpbvcon}, it follows that $f^+$ and $f^-$ have nontangential boundary values $f^+_{nt}$ and $f^-_{nt}$, respectively, $f^+_{nt}, f^-_{nt} \in L^p(\p D)$, 
\[
\lim_{r \nearrow 1} \int_0^{2\pi} |f^+_{nt}(e^{i\theta}) - f^+(re^{i\theta}) |^p \, d\theta  = 0,
\]
and 
\[
\lim_{r \nearrow 1} \int_0^{2\pi} |f^-_{nt}(e^{i\theta}) - f^-(re^{i\theta}) |^p \, d\theta  = 0.
\]
By linearity, $f_{nt} = p^+ f^+_{nt} + p^- f^-_{nt}$ exists and is in $L^p(\p D)$. Observe that, for each $r \in (0,1)$, 
\begin{align*}
&\int_0^{2\pi} ||f_{nt}(e^{i\theta}) - f(re^{i\theta}) ||^p_{\bc} \, d\theta \\
&= \int_0^{2\pi} ||p^+ f^+_{nt}(e^{i\theta}) + p^- f^-_{nt}(e^{i\theta}) - (p^+ f^+(re^{i\theta}) + p^- f^-(re^{i\theta})) ||^p_{\bc} \, d\theta \\
&= \int_0^{2\pi} ||p^+ (f^+_{nt}(e^{i\theta})-f^+(re^{i\theta})) + p^- (f^-_{nt}(e^{i\theta}) - f^-(re^{i\theta})) ||^p_{\bc} \, d\theta \\
&\leq \int_0^{2\pi} \left(\frac{1}{\sqrt{2}}(|f^+_{nt}(e^{i\theta})-f^+(re^{i\theta})| + |f^-_{nt}(e^{i\theta}) - f^-(re^{i\theta}) |)\right)^p \, d\theta \\
&\leq \frac{C_p}{2^{p/2}}\left(\int_0^{2\pi} |f^+_{nt}(e^{i\theta})-f^+(re^{i\theta})|^p \,d\theta + \int_0^{2\pi} |f^-_{nt}(e^{i\theta}) - f^-(re^{i\theta}) |^p \, d\theta \right), 
\end{align*}
where $C_p$ is a constant that depends only on $p$. Since 
\[
\lim_{r \nearrow 1} \int_0^{2\pi} |f^+_{nt}(e^{i\theta}) - f^+(re^{i\theta}) |^p \, d\theta  = 0
\]
and 
\[
\lim_{r \nearrow 1} \int_0^{2\pi} |f^-_{nt}(e^{i\theta}) - f^-(re^{i\theta}) |^p \, d\theta  = 0,
\]
it follows that 
\[
\lim_{r \nearrow 1}\frac{C_p}{2^{p/2}}\left(\int_0^{2\pi} |f^+_{nt}(e^{i\theta})-f^+(re^{i\theta})|^p \,d\theta + \int_0^{2\pi} |f^-_{nt}(e^{i\theta}) - f^-(re^{i\theta}) |^p \, d\theta \right) = 0.
\]
Therefore, 
\[
\lim_{r \nearrow 1} \int_0^{2\pi} ||f_{nt}(e^{i\theta}) - f(re^{i\theta}) ||^p_{\bc} \, d\theta = 0.
\]

\end{proof}

\section{Generalizations of the Bicomplex Hardy Spaces}\label{generalizedbchardysection}

In this section, we define generalizations of $H^p(D,\bc)$ that mimic the generalized Hardy classes $H^p_w(D)$ of \cite{WB}, see Definition \ref{cgenhardydef}. Specifically, we consider bicomplex-valued functions with $\dbar$-derivative not necessarily identically equal to zero and finite $H^p_\bc$ norm. 

\subsection{Definition}

\begin{deff}
    Let $0 < p < \infty$. We define the bicomplex Hardy classes $\hwp$ to be those functions $f \in \hw$ such that $||f||_{H^p_{\bc}} < \infty$.
\end{deff}

\subsection{Representation}

In this subsection, we show that bicomplex-valued solutions to nonhomogeneous $\dbar$-equations have a ``representation of the second kind,'' as named in \cite{Vek, WB}, similar to the $\mathbb{C}$ case. This shows that solutions to nonhomogeneous $\dbar$-equations have the structure of a sum of a $\bc$-holomorphic function and a well-controlled error term.

\begin{prop}\label{genbcnonhomogrep}
    Every solution $w: D \to \bc$ of 
    \[
        \dbar w = g,
    \]
    with $\dbar w = g \in L^1(D, \bc)$, is representable as 
    \[
        w = \varphi + T_\bc(g),
    \]
    where $\varphi \in \holb$. 
\end{prop}

\begin{proof}
Since $\dbar T_\bc (f) = f$, for every $ f\in L^1(D, \bc)$, it follows that, for any solution $w$ of $\dbar w = g$ such that $\dbar w = g\in L^1(D, \bc)$,
\[
    \dbar (w - T_\bc(g)) = g - g = 0.
\]
\end{proof}

Next, we prove the bicomplex analogue of Corollary \ref{genhardyrepcorr}. This provides a representation for functions $H^p_w(D,\bc)$ in terms of $H^p(D,\bc)$ functions. 

\begin{theorem}\label{generalbchardyrep}
    Let $0 < p < \infty$ and $w \in L^q(D,\bc)$, $q>2$. For every $f \in \hw$ with representation 
    \[
        f = \varphi + T_\bc(w),
    \]
    $ f \in \hwp$ if and only if $\varphi \in \hpb$. The result holds when $1 < q \leq 2$, so long as $p$ satisfies $p < \frac{q}{2-q}$. 
\end{theorem}

\begin{proof}
By hypothesis, $\dbar f= w \in L^1(D, \bc)$, so by Proposition \ref{genbcnonhomogrep}, 
\[
        f = \varphi + T_\bc(w).
\]
Suppose that $\varphi \in H^p(D,\bc)$. Then, for $r \in (0,1)$, we have
\begin{align}
 \int_0^{2\pi} ||f(re^{i\theta})||_{\bc}^p \,d\theta 
 &= \int_0^{2\pi} ||\varphi(re^{i\theta}) + T_\bc(w)(re^{i\theta})||_{\bc}^p \,d\theta \nonumber \\ 
 &\leq C_p\left( \int_0^{2\pi} ||\varphi(re^{i\theta})||_{\bc}^p + \int_0^{2\pi} ||T_\bc(w)(re^{i\theta})||_{\bc}^p \,d\theta \right) \nonumber\\
  &\leq C_p\left( ||\varphi||_{H^p_\bc}^p + \int_0^{2\pi} ||T_\bc(w)(re^{i\theta})||_{\bc}^p \,d\theta \right), \label{returnpoitbcvekhardyagain}
\end{align}
where $C_p$ is a constant that depends on only $p$. By Theorem \ref{bctbehavior}, since $ w\in L^q(D,\bc)$, $q>2$, it follows that $T_\bc(w) \in C^{0,\alpha}(\overline{D},\bc)$. Thus, there exists $M>0$ such that 
\[
        ||T_\bc(w)(z)||_{\bc} \leq M,
\]
for every $z \in \overline{D}$. Hence, 
\begin{align*}
C_p\left( ||\varphi||_{H^p_\bc}^p + \int_0^{2\pi} ||T_\bc(w)(re^{i\theta})||_{\bc}^p \,d\theta \right)
&\leq C_p\left( ||\varphi||_{H^p_\bc}^p + 2\pi M \right)< \infty.
\end{align*}
Since there is no dependence on $r$ in the right hand side of the above inequality, we have 
\begin{align*}
||f||_{H^p_\bc}^p &\leq C_p\left( ||\varphi||_{H^p_\bc}^p + 2\pi M \right)< \infty.
\end{align*}
If $1 < q \leq 2$ and $p$ satisfies $p < \frac{q}{2-q}$, then returning to (\ref{returnpoitbcvekhardyagain}) we have
\begin{align*}
C_p\left( ||\varphi||_{H^p_\bc}^p + \int_0^{2\pi} ||T_\bc(w)(re^{i\theta})||_{\bc}^p \,d\theta \right)
&\leq C_p\left( ||\varphi||_{H^p_\bc}^p + C||w||^p_{L^q(D)} \,d\theta \right) < \infty,
\end{align*}
by Theorem \ref{bctbehavior}. Since the right hand side has no dependence on $r$, we have
\[
||f||_{H^p_\bc}^p \leq C_p\left( ||\varphi||_{H^p_\bc}^p + C||w||^p_{L^q(D)} \,d\theta \right) < \infty.
\]
In either case $f \in H^p_{w}(D,\bc)$. 

Now, suppose that $f = \varphi + T_\bc(w) \in H^p_{w}(D,\bc)$. Observe that, for $r \in (0,1)$, we have
\begin{align}
    \int_0^{2\pi} ||\varphi(re^{i\theta})||_\bc^p \,d\theta 
    &= \int_0^{2\pi} ||f(re^{i\theta}) - T_\bc(w)(re^{i\theta})||_\bc^p \,d\theta \nonumber \\
    &\leq C_p \left( \int_0^{2\pi} ||f(re^{i\theta})||_{\bc}^p \,d\theta + \int_0^{2\pi} || T_\bc(w)(re^{i\theta})||_\bc^p \,d\theta\right)  \label{returnpointagainagain}\\
    &\leq C_p\left( ||f||_{H^p_\bc}^p + M 2\pi\right)< \infty,\nonumber
\end{align}
where $C_p$ is a constant that depends on only $p$ and $M$ is a constant that bounds $||T_\bc(w)(z)||_\bc^p$, for all $z \in \overline{D}$. So, 
\[
||\varphi||_{H^p_\bc}^p \leq C_p\left( ||f||_{H^p_\bc}^p + M 2\pi\right)< \infty,
\]
and $\varphi \in H^p(D, \bc)$. If $1 < q \leq 2$ and $p$ satisfies $p < \frac{q}{2-q}$, then returning to (\ref{returnpointagainagain}) and appealing to Theorem \ref{bctbehavior} once more, we have 
\begin{align*}
&C_p \left( \int_0^{2\pi} ||f(re^{i\theta})||_{\bc}^p \,d\theta + \int_0^{2\pi} || T_\bc(w)(re^{i\theta})||_\bc^p \,d\theta\right) \\
&\leq C_p \left( ||f||^p_{H^p_\bc}  + C||w||^p_{L^q(D)}\right)  < \infty.
\end{align*}
Therefore, 
\[
    ||\varphi||^p_{H^p_\bc} \leq C_p \left(  ||f||^p_{H^p_\bc} + C||w||^p_{L^q(D)}\right)  < \infty,
\]
and $\varphi \in H^p(D,\bc)$. 

\end{proof}

\subsection{Boundary Behavior}

With the representation theorem from the last subsection and the previous results about the boundary behavior for the $\bc$-holomorphic Hardy spaces $H^p(D,\bc)$, we are ready to show that the generalized bicomplex Hardy classes $H^p_w(D,\bc)$ have the desirable boundary behavior associated with Hardy classes. 

\begin{theorem}\label{Thm: generalbchardybv}
    For $0 < p < \infty$ and $w \in L^q(D,\bc)$, $q>2$, every $f = \varphi + T_\bc(w)\in \hwp$ such that $\varphi^+ \in H^p_{g}(D)$, where $g \in L^\gamma(D)$, $\gamma>2$ or $1 < \gamma \leq 2$ and $p < \frac{\gamma}{2-\gamma}$, has a nontangential limit $f_{nt} \in L^p(D, \bc)$, and 
    \[
        \lim_{r \nearrow 1} \int_0^{2\pi} ||f_{nt}(e^{i\theta}) - f(re^{i\theta}) ||^p_{\bc} \, d\theta  = 0. 
    \]  
    The result holds when $1 < q \leq 2$, so long as $p$ satisfies $p < \frac{q}{2-q}$.
\end{theorem}

\begin{proof}
By Theorem \ref{generalbchardyrep}, if $f \in H^p_{w}(D,\bc)$ with $w \in L^q(D)$, where $q$ and $p$ satisfy the hypothesis, then $f = \varphi + T_\bc(w)$ and $\varphi \in H^p(D,\bc)$. If we assume that $\varphi^+ \in H^p_g(D)$, for some $g \in L^\gamma(D)$, $\gamma>2$ or $1 < \gamma \leq 2$ and $p < \frac{\gamma}{2-\gamma}$, then by Theorem \ref{bcholobvcon}, $\varphi$ has a nontangential boundary value $\varphi_{nt} \in L^p(\p D,\bc)$ and 
\[
        \lim_{r \nearrow 1} \int_0^{2\pi} ||\varphi_{nt}(e^{i\theta}) - \varphi(re^{i\theta}) ||^p_{\bc} \, d\theta  = 0. 
    \]
By Theorem \ref{bctbehavior}, if $p$ and $q$ satisfy the hypothesis, then $T_\bc(w) \in L^p(\p D,\bc)$ and 
\[
\lim_{r \nearrow 1} \int_0^{2\pi} ||T_\bc(w)(e^{i\theta}) - T_\bc(w)(re^{i\theta}) ||^p_{\bc} \, d\theta  = 0.
\]
So, $f_{nt}$ exists and is in $L^p(\p D,\bc)$. Also, for $r \in (0,1)$, we have
\begin{align*}
    &\int_0^{2\pi} ||f_{nt}(e^{i\theta}) - f(re^{i\theta}) ||^p_{\bc} \, d\theta \\
    &= \int_0^{2\pi} ||\varphi_{nt}(e^{i\theta}) + T_\bc(w)(e^{i\theta}) - (\varphi(re^{i\theta}) + T_\bc(w)(re^{i\theta})) ||^p_{\bc} \, d\theta \\
    &\leq C_p \left( \int_0^{2\pi} ||\varphi_{nt}(e^{i\theta}) - \varphi(re^{i\theta})||_\bc^p\,d\theta  + \int_0^{2\pi} || T_\bc(w)(e^{i\theta}) - T_\bc(w)(re^{i\theta}) ||^p_{\bc} \, d\theta\right) .
\end{align*}
Therefore, 
\begin{align*}
& \lim_{r \nearrow 1} \int_0^{2\pi} ||f_{nt}(e^{i\theta}) - f(re^{i\theta}) ||^p_{\bc} \, d\theta \\
&\leq \lim_{r \nearrow 1} C_p \left( \int_0^{2\pi} ||\varphi_{nt}(e^{i\theta}) - \varphi(re^{i\theta})||_\bc^p\,d\theta  + \int_0^{2\pi} || T_\bc(w)(e^{i\theta}) - T_\bc(w)(re^{i\theta}) ||^p_{\bc} \, d\theta\right)\\
&\leq  C_p \left( \lim_{r \nearrow 1}\int_0^{2\pi} ||\varphi_{nt}(e^{i\theta}) - \varphi(re^{i\theta})||_\bc^p\,d\theta  + \lim_{r \nearrow 1}\int_0^{2\pi} || T_\bc(w)(e^{i\theta}) - T_\bc(w)(re^{i\theta}) ||^p_{\bc} \, d\theta\right)\\
&= 0.
\end{align*}

\end{proof}

\section{Atomic Decomposition}\label{atomicdecompsection}

In this section, we show that the bicomplex Hardy classes have boundary values in the sense of distributions and distributional boundary values of functions in the bicomplex Hardy classes, for small $p$, have an atomic decomposition. From this atomic decomposition, we show that the classic Hilbert transform on the circle is a continuous operator on this class of boundary distributions. This extends the results in the $\mathbb{C}$-holomorphic case from \cite{GHJH2} and the $\mathbb{C}$-valued but not necessarily holomorphic  case from \cite{WB}.

    By the linearity of the integral, $w = p^+ w^+ + p^- w^- : D \to \bc$ has a boundary value in the sense of distributions whenever the functions $w^+,w^- D: \to \mathbb{C}$ have boundary values in the sense of distributions. We begin by showing that functions in $\hpb$ have boundary values in the sense of distributions.

\begin{prop}\label{Prop: bcholodistbv}
    For $0 < p \leq 1$, every $w = p^+ w^+ + p^- w^-\in \hpb$ has a distributional boundary value $w_b$ and 
    \[
        w_b = p^+ w^+_b + p^- w^-_b.
    \]  
\end{prop}

\begin{proof}
    By Theorem \ref{thm: bchardyrep}, if $w \in \hpb$, then $(w^+)^*, w^- \in H^p(D)$. By Theorem \ref{GHJH23point1} and \ref{GHJH23point1corr}, since $(w^+)^*, w^- \in H^p(D)$, it follows that $(w^+)^*_b, w^-_b$, the distributional boundary values of $(w^+)^*$ and $w^-$, respectively, exist. 
    
    Since $(w^+)^*_b$ exists, it follows, by definition, that the limit
    \[
        \lim_{r\nearrow 1} \int_0^{2\pi} (w^+)^*(re^{i\theta}) \,\varphi(e^{i\theta})\,d\theta
    \]
    exists, for every $\varphi \in C^\infty(\p D)$. Since, for every $r \in (0,1)$ and $\varphi \in C^\infty(\p D)$, 
    \[
        \int_0^{2\pi} (w^+)^*(re^{i\theta}) \,\varphi(e^{i\theta})\,d\theta 
        = \left( \int_0^{2\pi} w^+(re^{i\theta}) \,\varphi^*(e^{i\theta})\,d\theta \right)^*,
    \]
    it follows that the limit
    \[
        \left( \lim_{r\nearrow 1}\int_0^{2\pi} w^+(re^{i\theta}) \,\varphi^*(e^{i\theta})\,d\theta \right)^* = 
       \lim_{r\nearrow 1} \int_0^{2\pi} (w^+)^*(re^{i\theta}) \,\varphi(e^{i\theta})\,d\theta 
    \]
    exists. Hence, $w^+_b$ exists.

    Now, observe that for $\varphi \in C^{\infty}(\p D)$ and $r \in (0,1)$, we have
    \begin{align*}
        \int_0^{2\pi} w(re^{i\theta}) \,\varphi(e^{i\theta})\,d\theta
        &=  p^+\int_0^{2\pi} w^+(re^{i\theta}) \,\varphi(e^{i\theta})\,d\theta + p^-\int_0^{2\pi} w^-(re^{i\theta}) \,\varphi(e^{i\theta})\,d\theta.
    \end{align*}
    Since the limits
    \[
        \lim_{r\nearrow 1} \int_0^{2\pi} w^+(re^{i\theta}) \,\varphi(e^{i\theta})\,d\theta 
    \]
    and
    \[
\lim_{r\nearrow 1} \int_0^{2\pi} w^-(re^{i\theta}) \,\varphi(e^{i\theta})\,d\theta 
    \]
    exist, it follows that the limit
    \begin{align*}
       & \lim_{r \nearrow 1}\int_0^{2\pi} w(re^{i\theta}) \,\varphi(e^{i\theta})\,d\theta \\
        &=  p^+\lim_{r \nearrow 1}\int_0^{2\pi} w^+(re^{i\theta}) \,\varphi(e^{i\theta})\,d\theta + p^-\lim_{r \nearrow 1}\int_0^{2\pi} w^-(re^{i\theta}) \,\varphi(e^{i\theta})\,d\theta
    \end{align*}
    exists. Thus, $w_b$ exists, and 
    \[
        w_b = p^+ w^+_b + p^- w^-_b.
    \]
\end{proof}

We are now ready to show the variant of the atomic decomposition result in this context. 

\begin{theorem}\label{Thm: bcholoatomicdecomp}
    For $0 < p \leq 1$ and every $w \in \hpb$, there exist sequences $\{c_n^+\}, \{c_n^-\} \in \ell^p(\C)$ and collections of $p$-atoms $\{a_n^+\}$ and $\{a_n^-\}$ such that 
    \[
        w_b = p^+ \left( \sum_{n=1}^\infty c^+_n a^+_n\right)^* + p^- \sum_{k=1}^\infty c^-_k a^-_k.
    \]
\end{theorem}

\begin{proof}
By Theorem \ref{thm: bchardyrep}, $w = p^+w^+ +p^- w^- \in \hpb$ if and only if $(w^+)^*, w^- \in H^p(D)$. By Theorem \ref{GHJH2twopointtwo}, since $(w^+)^*,w^- \in H^p(D)$, it follows that there exists $\{c_n^+\}, \{c_n^-\} \in \ell^p(\C)$ and collections of $p$-atoms $\{a_n^+\}$ and $\{a_n^-\}$
such that 
\[
    (w^+)^*_b = \sum_{n=1}^\infty c_n^+ a_n^+
\]
and 
\[
    w^-_b = \sum_{n=1}^\infty c_n^- a_n^-.
\]
By Proposition \ref{Prop: bcholodistbv}, $w_b$ exists and 
\[
    w_b = p^+ w^+_b + p^- w^-_b = p^+((w^+)^*_b)^* + p^- w^-.
\]
Thus, 
\[
    w_b = p^+ \left( \sum_{n=1}^\infty c_n^+a_n^+ \right)^* + p^- \sum_{n=1}^\infty c_n^-a_n^-.
\]

\end{proof}

In pursuit of showing that the Hilbert transform is a continuous operator on the set of boundary values in the sense of distributions of functions in $\hpb$, for small $p$, we require a norm for this class of objects. We define the one that we consider. 

\begin{deff}
    For $0 < p \leq 1$, we define $(\hpb)_b$ to be the collection of distributional boundary values of functions in $\hpb$. 
\end{deff}

\begin{deff}\label{bcatomicnorm}
    We define the $\bc$-atomic norm $||\cdot||_{\bc,at}$ on $(\hpb)_b$ by 
    \[
        ||w_b||_{\bc,at} := ||w^+_b||_{at} + ||w^-_b||_{at},
    \]
    where $||\cdot||_{at}$ is the atomic norm defined on the distributional boundary values of functions in the holomorphic Hardy spaces $H^p(D)$ defined in Definition \ref{atomicnorm}.
\end{deff}

\begin{theorem}\label{bcholohilbertcont}
    For $0 < p \leq 1$, the Hilbert transform is a continuous operator on $(\hpb)_b$ with the norm $||\cdot||_{\bc,at}$. 
\end{theorem}

\begin{proof}
By Theorem \ref{thm: bchardyrep}, if $w \in \hpb$, $0 < p \leq 1$, then $w = p^+ w^+ + p^- w^-$ and $(w^+)^*, w^- \in H^p(D)$. By Theorem \ref{hilbertholohardy}, the Hilbert transform is a continuous operator on the distributional boundary values of the functions in $H^p(D)$, $0 < p \leq 1$, with the atomic norm i.e., for every $f \in (H^p(D))_b$, there exists a constant $C>0$ such that 
\[
    ||H(f)||_{L^p(\p D)} \leq C ||f||_{at}. 
\]
Since $(w^+)^*,w^- \in H^p(D)$, it follows that $(w^+)_b, w^-_b$ exist and there exist positive constants $C_1$ and $C_2$ such that 
\[
||H((w^+)^*_b)||_{L^p(\p D)} \leq C_1 ||(w^+)^*_b||_{at}
\]
and 
\[
||H(w^-_b)||_{L^p(\p D)} \leq C_2 ||w^-_b||_{at}.
\]
Thus, 
\begin{align*}
    ||H(w_b)||_{L^p(\p D, \bc)} 
    &\leq \frac{1}{\sqrt{2}}\left(  ||H(w^+)||_{L^p(\p D)} + ||H(w^-)||_{L^p(\p D)} \right) .
\end{align*}
Since
\[
    ||H(w^+)||_{L^p(\p D)} = ||(H((w^+)^*))^*||_{L^p(\p D)} = ||H((w^+)^*)||_{L^p(\p D)} 
\]
and
\[
        ||(w^+)^*_b||_{at} = ||w^+_b||_{at},
\]
it follows that 
\begin{align*}
\frac{1}{\sqrt{2}}\left(  ||H(w^+)||_{L^p(\p D)} + ||H(w^-)||_{L^p(\p D)} \right) 
&= \frac{1}{\sqrt{2}}\left(  ||H((w^+)^*)||_{L^p(\p D)} + ||H(w^-)||_{L^p(\p D)} \right) \\
&\leq \frac{1}{\sqrt{2}}\left(  C_1 ||(w^+)^*_b||_{at} + C_2 ||w^-_b||_{at} \right) \\
&= \frac{1}{\sqrt{2}}\left(  C_1 ||w^+_b||_{at} + C_2 ||w^-_b||_{at} \right) \\
& < C \left( ||w^+_b||_{at} +  ||w^-_b||_{at}\right) \\
&= C ||w_b||_{\bc,at},
\end{align*}
where $C$ is a constant that is larger than $2^{-1/2}C_1$ and $2^{-1/2}C_2$.

\end{proof}

We are able to immediately extend from the holomorphic to the generalized setting as in the complex setting of \cite{WB}. 

\begin{theorem}\label{bcgeneralhardyatomicdecomp}
Let $0 < p \leq 1$ and $ w \in L^q(D,\bc)$, $q>1$. For every $f \in H^p_w(D, \bc)$, there exist sequences $\{c_n^+\}, \{c_n^-\} \in \ell^p(\C)$ and collections of $p$-atoms $\{a_n^+\}$ and $\{a_n^-\}$ such that 
    \[
        f_b = p^+ \left( \sum_{n=1}^\infty c^+_n a^+_n\right)^* + p^- \sum_{k=1}^\infty c^-_k a^-_k + T_\bc(w)_b,
    \]
    and $T_\bc(w)_b \in L^\gamma(\p D,\bc)$, $1 < \gamma< \frac{q}{2-q}$. If $q>2$, then $T_\bc(w)_b \in C^{0,\alpha}(\p D,\bc)$, $\alpha = \frac{q-2}{q}$. 
\end{theorem}

\begin{proof}
For $0 < p \leq 1$ and $w \in L^q(D)$, $q>1$, if $f \in H^p_w(D,\bc)$, then, by Theorem \ref{generalbchardyrep},
\[
    w = \varphi + T_\bc(w),
\]
and $\varphi \in H^p(D,\bc)$. By Theorem \ref{Thm: bcholoatomicdecomp}, $\varphi_b$ exists, and there exist $\{c_n^+\}, \{c_n^-\} \in \ell^p(\C)$ and collections of $p$-atoms $\{a_n^+\}$ and $\{a_n^-\}$ such that
\[
    \varphi_b = p^+ \left( \sum_{n=1}^\infty c^+_n a^+_n\right)^* + p^- \sum_{k=1}^\infty c^-_k a^-_k.
\]

By Theorem \ref{bctbehavior}, since $q >1$, it follows that $T_\bc(w)|_{\p D} \in L^\gamma(\p D, \bc)$, $1 < \gamma< \frac{q}{2-q}$, and if $q>2$, then $T_\bc(w)|_{\p D} \in C^{0,\alpha}(\p D,\bc)$, $\alpha = \frac{q-2}{q}$. Since $q>1$, it follows, by Theorem \ref{bctbehavior}, that $T_\bc(w) \in L^1(D)$ and 
\[
    \lim_{r \nearrow 1} \int_0^{2\pi} ||T_\bc(w)(e^{i\theta})- T_\bc(w)(re^{i\theta})||_{\bc}\,d\theta = 0
\]
so by Theorem \ref{lonedistbv}, $T_\bc(w)_b$ exists and $T_\bc(w)_b = T_\bc(w)|_{\p D}$. 

Since $f = \varphi + T_\bc(w)$ and $\varphi_b$ and $T_\bc(w)_b$ exist, it follows that $f_b$ exists and 
\[
    f_b = \varphi_b + T_\bc(w)_b.
\]
Since 
\[
    \varphi_b = p^+ \left( \sum_{n=1}^\infty c^+_n a^+_n\right)^* + p^- \sum_{k=1}^\infty c^-_k a^-_k,
\]
it follows that 
\[
        f_b = p^+ \left( \sum_{n=1}^\infty c^+_n a^+_n\right)^* + p^- \sum_{k=1}^\infty c^-_k a^-_k + T_\bc(w)_b.
\]
\end{proof}

Also, we show the Hilbert transform is a continuous operator.

\begin{deff}
    For $0 < p \leq 1$ and $w \in L^q(D)$, $q>1$, we define $(H^p_w(D,\bc))_b$ to be the collection of distributional boundary values of functions in $H^p_w(D,\bc)$. 
\end{deff}

\begin{deff}
    We define the quasi-norm $||\cdot||_{\bc,b}$ on $f_b \in (H^p_w(D,\bc))_b$ by 
    \[
        ||f_b||_{\bc,b} := ||\varphi_b||_{\bc,at} + ||T_\bc(w)_b||_{L^\gamma(\p D)},
    \]
    where $||\cdot||_{\bc,at}$ is the $\bc$-atomic norm on the distributional boundary values of functions in the $\bc$-holomorphic Hardy spaces $H^p(D,\bc)$ defined in Definition \ref{bcatomicnorm}.
\end{deff}

\begin{theorem}\label{Thm: bcnonhomoghhilbertrans}
    For $0 < p \leq 1$ and $w \in L^q(D,\bc)$, $q>1$, the Hilbert transform is a continuous operator on $(H^p_w(D,\bc))_b$ with the quasi-norm $||\cdot||_{\bc,b}$. 
\end{theorem}

\begin{proof}
Let $f \in H^p_w(D, \bc)$. By Theorem \ref{bcgeneralhardyatomicdecomp}, 
\[
    f_b = \varphi_b + T_\bc(w)_b.
\]
By Theorem \ref{bcholohilbertcont}, since $\varphi_b \in H^p(D, \bc)$ and $0 < p \leq 1$, it follows that there exists a positive constant $C_1$ such that 
\[
    ||H(\varphi_b)||_{L^p(\p D)} \leq C_1 ||\varphi_b||_{\bc,at}.
\]
By Theorem \ref{hilbertpgreaterthanone}, since $T_\bc(w)_b \in L^\gamma(\p D)$ and $\gamma > 1$, by Theorem \ref{bctbehavior}, it follows that there exist a positive constant $C_2$ such that 
\[
    ||H(T_\bc(w)_b)||_{L^p(\p D)} \leq C_2 ||T_\bc(w)_b||_{L^p(\p D)}.
\]
Therefore, 
\begin{align*}
    ||H(f_b)||_{L^p(\p D)} 
    &\leq ||H(\varphi_b)||_{L^p(\p D)} + ||H(T_\bc(w)_b)||_{L^p(\p D)}\\
    &\leq C_1 ||\varphi_b||_{\bc,at} + C_2 ||T_\bc(w)_b||_{L^\gamma(\p D)} \\
    &\leq C\left( ||\varphi_b||_{\bc,at} +  ||T_\bc(w)_b||_{L^\gamma(\p D)} \right)\\
    &= C||f_b||_{\bc,b},
\end{align*}
where $C$ is a constant that dominates $C_1$ and $C_2$. 

\end{proof}

\section{Higher-Order Generalizations of the Bicomplex Hardy Spaces}\label{higherordersection}

In this section, we extend the results of the previous sections to higher-order variants of the bicomplex Hardy classes considered. This mirrors the progression from first-order to higher-order in \cite{WB} for the complex-valued Hardy classes. 

\subsection{Definition}

\begin{deff}
    Let $0 < p < \infty$, $n$ a positive integer, and $w : D\to\bc$. We define $H^{n,p}_{w}(D, \bc)$ to be the set of functions $f: D \to\bc$ such that 
    \[
        \dbar^n f = w
    \]
    and 
    \[
        \sum_{k=0}^{n-1} ||\dbar^k f ||^p_{H^p_\bc} < \infty. 
    \]
\end{deff}

\subsection{Representation}

First, we prove a result that is both classical in the complex holomorphic setting and critical to the proof of the representation result that follows.

\begin{prop}\label{Prop: bchardymlessthan2p}
For $0 < p < \infty$, every $w \in H^p(D,\bc)$ is an element of $L^m(D,\bc)$, for all $0 < m < 2p$. 
\end{prop}

\begin{proof}
Recall that, for $0 < p < \infty$, if $g \in H^p(D)$, then $g \in L^m(D)$, for every $0 < m < 2p$. See, for example, Lemma 1.8.3 of \cite{KlimBook}. 

Now, by Theorem \ref{thm: bchardyrep}, if $f \in H^p(D, \bc)$, then $(f^+)^*,f^- \in H^p(D)$. So, $(f^+)^*,f^- \in L^m(D)$, for all $0 < m < 2p$. Since 
\[
        |(f^+)^*(z)| = |f^+(z)|,
\]
for all $z \in D$, it follows that $f^+ \in L^m(D)$, for all $0 < m < 2p$. By Proposition \ref{propLqiff}, since $f^+, f^- \in L^m(D)$, for all $0 < m < 2p$, it follows that $f \in L^m(D,\bc)$, for all $0 < m < 2p$. 

\end{proof}

We are now ready to prove the analogue of the pointwise function representation from Theorem 4.3 of \cite{WB} for the complex case. Existence and representation of distributional boundary values will be handled in the next section, along with nontangential boundary values. 

\begin{theorem}\label{Thm: bchigherhardyrep}
    Let $1< p < \infty$, $n$ a positive integer, and $w \in L^q(D,\bc)$, $q>2$. Every $f \in H^{n,p}_{w}(D,\bc)$ can be represented as 
    \[
        f = \Phi_0 + \Psi,
    \]
    where
    \[
        \Psi = T_\bc( \Phi_1 + T_\bc( \Phi_2 + T_\bc( \cdots + \Phi_{n-1} + T_\bc(w))\cdots) )
    \]
    and $\Phi_k \in H^p(D,\bc)$, for every $k$, and $\Psi \in C^{0,\alpha}(\overline{D},\bc)$. The result holds for $p > \frac{1}{2}$ and $1 < q \leq 2$, with $\Psi \in L^\gamma(D,\bc)$, $1 < \gamma< \frac{2q}{2-q}$, so long as $p$ satisfies $\frac{1}{2} < p < \frac{q}{2-q}$. 
    
\end{theorem}

\begin{proof}
First, suppose $p > 1$ and $q>2$. By applying Theorem \ref{generalbchardyrep}, we have
\[
    \dbar^{n-1}f = \Phi_{n-1} + T_\bc(w),
\]
where $\Phi_{n-1} \in H^p(D,\bc)$ and $T_\bc(w) \in C^{0,\alpha}(\overline{D},\bc)$. Since $p>1$, it follows by Proposition \ref{Prop: bchardymlessthan2p} that $\Phi_{n-1} \in L^{q_{n-1}}(D,\bc)$, for some $2 < q_{n-1} < 2p$. Hence, $\Phi_{n-1} + T_\bc(w) \in L^{q_{n-1}}(D,\bc)$. So, we appeal to Theorem \ref{generalbchardyrep} again to have
\[
    \dbar^{n-2}f = \Phi_{n-2} + T_\bc(\Phi_{n-1} + T_\bc(w)).
\]
Repeating this $n-2$ more times, we have the representation
\[
        f = \Phi_0 + \Psi,
    \]
    where
    \[
        \Psi = T_\bc( \Phi_1 + T_\bc( \Phi_2 + T_\bc( \cdots + \Phi_{n-1} + T_\bc(w))\cdots) ).
    \]

Now, suppose that $1< q \leq 2$. So long as $p < \frac{q}{2-q}$, we have 
\[
    \dbar^{n-1}f = \Phi_{n-1} + T_\bc(w),
\]
by Theorem \ref{generalbchardyrep}, where $\Phi_{n-1} \in H^p(D, \bc)$ and $T_\bc(w) \in L^\gamma(D,\bc)$, $1 < \gamma < \frac{2q}{2-q}$. If $p > \frac{1}{2}$, then, by Proposition \ref{Prop: bchardymlessthan2p}, $\Phi_{n-1} \in L^{q_{n-1}}(D,\bc)$, where $1 < q_{n-1} < 2p$. So, $\Phi_{n-1} + T_\bc(w) \in L^{q_{n-1}}(D,\bc)$. So, we appeal to Theorem \ref{generalbchardyrep} again to have
\[
    \dbar^{n-2}f = \Phi_{n-2} + T_\bc(\Phi_{n-1} + T_\bc(w)).
\]
Repeating this $n-2$ more times, we have the representation
\[
        f = \Phi_0 + \Psi,
    \]
    where
    \[
        \Psi = T_\bc( \Phi_1 + T_\bc( \Phi_2 + T_\bc( \cdots + \Phi_{n-1} + T_\bc(w))\cdots) ).
    \]

Therefore, the result holds for $p>1$ and $q>2$, and the result holds for $1 < q \leq 2$, so long as $p$ satisfies $\frac{1}{2} < p < \frac{q}{2-q}$. 

\end{proof}

	\begin{corr}
		  Let $1< p < \infty$, $n$ a positive integer, and $w \in L^q(D,\bc)$, $q>2$. Every $f \in H^{n,p}_{w}(D,\bc)$ can be represented as 
		  \begin{align*}
		  	f(z) &= \Phi_0(z) + p^+\left[\sum_{k=1}^{n-1}   \left(\iint_D K_{k,0}(z-\zeta) \Phi^+_{k}(\zeta)\,d\eta\,d\xi \right) + \iint_D K_{n,0}(z-\zeta) w^+(\zeta)\,d\eta\,d\xi \right] \\
			&\quad \quad+  p^-\left[ \sum_{k=1}^{n-1} \left(\iint_{D}K_{0,k}(z-\zeta) \Phi_k^-(\zeta)\,d\eta\,d\xi \right)+ \iint_{D}K_{0,n}(z-\zeta) w^-(\zeta)\,d\eta\,d\xi\right],
		  \end{align*}
		   where, for integers $m$ and $\gamma$,
		   \[
		   	K_{m,\gamma}(z) := \begin{cases}
						\frac{(-m)! (-1)^m}{(\gamma-1)! \pi} z^{m-1} (z^*)^{\gamma-1}, & m = 0, \\
						\frac{(-\gamma)! (-1)^\gamma}{(m-1)! \pi} z^{m-1} (z^*)^{\gamma-1},  &  \gamma = 0,\\
						 \frac{1}{(m-1)! (\gamma-1)! \pi} z^{m-1} (z^*)^{\gamma-1} \left( \log|z|^2 - \sum_{k=1}^{m-1} \frac{1}{k} - \sum_{\ell = 1}^{\gamma-1} \frac{1}{\ell}\right), & m, \gamma \geq 1, \\
			\end{cases}
		   \]
		  (the first or second sums in the last case are taken to be zero if $m =1$ or $\gamma =1$, respectively),  and $\Phi_k \in H^p(D,\bc)$, for every $k$. The result holds for $p > \frac{1}{2}$ and $1 < q \leq 2$, so long as $p$ satisfies $\frac{1}{2} < p < \frac{q}{2-q}$. 
	\end{corr}
	
	\begin{proof}
		This representation is a direct computation using the definition of $T_\bc$ combined with the result of Theorem \ref{Thm: bchigherhardyrep} and the work of Begehr and Hile in \cite{hoio} (or Begehr in \cite{BegBook}).
	\end{proof}

\subsection{Boundary Behavior}

The next theorem shows the bicomplex higher-order Hardy classes share the boundary behavior we expect of a Hardy class. 

\begin{theorem}
    Let $1< p < \infty$, $n$ a positive integer, and $w \in L^q(D,\bc)$, $q>2$. Every $f = \Phi_0 + \Psi \in H^{n,p}_w(D,\bc)$ such that $\Phi_0^+ \in H^p_g(D)$, where $g \in L^\gamma(D)$, and $\gamma>2$ or $1 < \gamma \leq 2$ and $p$ satisfies $p < \frac{\gamma}{2-\gamma}$, has a nontangential limit $f_{nt} \in L^p(\p D, \bc)$, and 
    \[
        \lim_{r \nearrow 1} \int_0^{2\pi} ||f_{nt}(e^{i\theta}) - f(re^{i\theta}) ||^p_{\bc} \, d\theta  = 0. 
    \] 
     The result holds for $p > \frac{1}{2}$ when $1 < q \leq 2$, so long as $p$ satisfies $\frac{1}{2} < p < \frac{q}{2-q}$.
\end{theorem}

\begin{proof}
If $p > 1$, $n$ a positive integer, $w \in L^q(D, \bc)$, $q>2$, and $w \in H^{n,p}_{w}(D, \bc)$, then, by Theorem \ref{Thm: bchigherhardyrep}, $f = \Phi_0 + \Psi$, where $\Phi_0 \in H^p(D,\bc)$ and $\Psi \in C^{0,\alpha}(\overline{D},\bc)$. If $\Phi_0^+ \in H^p_g(D,\bc)$ and $g \in L^\gamma(D)$ and $\gamma > 2$ or $1 < \gamma \leq 2$ and $p$ satisfies $p < \frac{\gamma}{2-\gamma}$, then, by Theorem \ref{bcholobvcon}, $\Phi_0$ has a nontangential boundary value $\Phi_{0,nt} \in L^p(\p D, \bc)$ and 
\[
        \lim_{r \nearrow 1} \int_0^{2\pi} ||\Phi_{0,nt}(e^{i\theta}) - \Phi_0(re^{i\theta}) ||^p_{\bc} \, d\theta  = 0. 
    \] 
So, the nontangential boundary value $f_{nt}$ of $f$ exists and equals $f_{nt} = \Phi_{0,nt} + \Psi|_{\p D}$. By Theorem \ref{bctbehavior},
\[
        \lim_{r \nearrow 1} \int_0^{2\pi} ||\Psi(e^{i\theta}) - \Psi(re^{i\theta}) ||^p_{\bc} \, d\theta  = 0. 
\]
Therefore, for $r \in (0,1)$, we have
\begin{align*}
    & \int_0^{2\pi} ||f_{nt}(e^{i\theta}) - f(re^{i\theta}) ||^p_{\bc} \, d\theta\\
    &= \int_0^{2\pi} ||\Phi_{0,nt}(e^{i\theta}) + \Psi(e^{i\theta}) - (\Phi_0(re^{i\theta}) + \Psi(re^{i\theta})) ||^p_{\bc} \, d\theta \\
    &\leq C_p \left( ||\Phi_{0,nt}(e^{i\theta}) - \Phi_0(re^{i\theta}) ||^p_{\bc} \, d\theta + \int_0^{2\pi} ||\Psi(e^{i\theta}) - \Psi(re^{i\theta}) ||^p_{\bc} \, d\theta \right),
\end{align*}
where $C_p$ is a constant that depends on only $p$. Thus, 
\begin{align*}
    &\lim_{r \nearrow 1} \int_0^{2\pi} ||f_{nt}(e^{i\theta}) - f(re^{i\theta}) ||^p_{\bc} \, d\theta \\
    &\leq \lim_{r \nearrow 1}C_p \left( ||\Phi_{0,nt}(e^{i\theta}) - \Phi_0(re^{i\theta}) ||^p_{\bc} \, d\theta + \int_0^{2\pi} ||\Psi(e^{i\theta}) - \Psi(re^{i\theta}) ||^p_{\bc} \, d\theta \right) = 0.
\end{align*}

If $1 < q \leq 2$ and $p$ satisfies $\frac{1}{2} < p < \frac{q}{2-q}$, then, by Theorem \ref{Thm: bchigherhardyrep}, $f = \Phi_0 + \Psi$, where $\Phi_0 \in H^p(D,\bc)$, $\Psi \in L^r(D, \bc)$, $1 < r < \frac{2q}{2-q}$, and $\Psi|_{\p D}(\p D, \bc) \in L^s(\p D, \bc)$, $1 < s < \frac{q}{2-q}$. So, if $\Phi_0^+ \in H^p_g(D,\bc)$ and $g \in L^\gamma(D)$ and $\gamma > 2$ or $1 < \gamma \leq 2$ and $p$ satisfies $p < \frac{\gamma}{2-\gamma}$, then, by Theorem \ref{bcholobvcon}, $\Phi_0$ has a nontangential boundary value $\Phi_{0,nt} \in L^p(\p D, \bc)$ and 
\[
        \lim_{r \nearrow 1} \int_0^{2\pi} ||\Phi_{0,nt}(e^{i\theta}) - \Phi_0(re^{i\theta}) ||^p_{\bc} \, d\theta  = 0. 
    \] 
So, the nontangential boundary value $f_{nt}$ of $f$ exists and equals $f_{nt} = \Phi_{0,nt} + \Psi|_{\p D}$. By Theorem \ref{bctbehavior},
\[
        \lim_{r \nearrow 1} \int_0^{2\pi} ||\Psi(e^{i\theta}) - \Psi(re^{i\theta}) ||^p_{\bc} \, d\theta  = 0, 
\]
as $p < \frac{q}{2-q}$ by assumption. Therefore, for $r \in (0,1)$, we have
\begin{align*}
    & \int_0^{2\pi} ||f_{nt}(e^{i\theta}) - f(re^{i\theta}) ||^p_{\bc} \, d\theta\\
    &= \int_0^{2\pi} ||\Phi_{0,nt}(e^{i\theta}) + \Psi(e^{i\theta}) - (\Phi_0(re^{i\theta}) + \Psi(re^{i\theta})) ||^p_{\bc} \, d\theta \\
    &\leq C_p \left( ||\Phi_{0,nt}(e^{i\theta}) - \Phi_0(re^{i\theta}) ||^p_{\bc} \, d\theta + \int_0^{2\pi} ||\Psi(e^{i\theta}) - \Psi(re^{i\theta}) ||^p_{\bc} \, d\theta \right),
\end{align*}
where $C_p$ is a constant that depends on only $p$. Thus, 
\begin{align*}
    &\lim_{r \nearrow 1} \int_0^{2\pi} ||f_{nt}(e^{i\theta}) - f(re^{i\theta}) ||^p_{\bc} \, d\theta \\
    &\leq \lim_{r \nearrow 1}C_p \left( ||\Phi_{0,nt}(e^{i\theta}) - \Phi_0(re^{i\theta}) ||^p_{\bc} \, d\theta + \int_0^{2\pi} ||\Psi(e^{i\theta}) - \Psi(re^{i\theta}) ||^p_{\bc} \, d\theta \right) = 0.
\end{align*}
\end{proof}

Also, we show the bicomplex higher-order Hardy classes have distributional boundary values that are representable by an atomic decomposition. 

\begin{theorem}\label{Thm: higheratomicbc}
Let $\frac{1}{2} < p \leq 1$, $n$ be a positive integer, and $w \in L^q(D,\bc)$, $q>1$.  Every $f \in H^{n,p}_{w}(D,\bc)$ has a boundary value in the sense of distributions $f_b$, and $f_b$ can be represented as 
    \[
        f_b = p^+ \left( \sum_{n=1}^\infty c^+_n a^+_n\right)^* + p^- \sum_{k=1}^\infty c^-_k a^-_k + \Psi_b,
    \]
    where $\{c_n^+\}, \{c_n^-\} \in \ell^p(\C)$, $\{a_n^+\}$ and $\{a_n^-\}$ are collections of $p$-atoms, and $\Psi_b \in L^\gamma(\p D,\bc)$, $1 < \gamma< \frac{q}{2-q}$. If $q>2$, then $\Psi_b \in C^{0,\alpha}(\p D,\bc)$, $\alpha = \frac{q-2}{q}$. 
\end{theorem}

\begin{proof}
By Theorems \ref{Thm: bchigherhardyrep} and \ref{bctbehavior}, every $f \in H^{n,p}_{w}(D,\bc)$ has a representation 
\[
    f = \Phi_0 + \Psi,
\]
where $\Phi_0 \in H^p(D,\bc)$ and $\Psi \in C^{0,\alpha}(\overline{D},\bc)$ or $L^\gamma(D,\bc)$, $1 < \gamma< \frac{2q}{2-q}$, depending on the value of $q$. Also, by Theorem \ref{bctbehavior}, $\Psi|_{\p D} \in L^\gamma(\p D)$, $1 < \gamma< \frac{2q}{q-2}$. By Theorem \ref{lonedistbv}, $\Psi_b$ exists and equals $\Psi|_{\p D}$. By Theorem \ref{Thm: bcholoatomicdecomp},  the distributional boundary value $\Phi_{0,b}$ of $\Phi_0$ exists, and there exist $\{c_n^+\}, \{c_n^-\} \in \ell^p(\C)$, $\{a_n^+\}$ and $\{a_n^-\}$ are collections of $p$-atoms such that 
\[
    \Phi_{0,b} = p^+ \left( \sum_{n=1}^\infty c^+_n a^+_n\right)^* + p^- \sum_{k=1}^\infty c^-_k a^-_k.
\]
Thus, the distributional boundary value $f_b$ exists, and 
\[
    f_b = p^+ \left( \sum_{n=1}^\infty c^+_n a^+_n\right)^* + p^- \sum_{k=1}^\infty c^-_k a^-_k + \Psi_b
\]
\end{proof}

With the atomic representation from the last theorem, we show the Hilbert transform is a continuous operator on the distributional boundary values of $H^{n,p}_w(D,\bc)$. 

\begin{theorem}
Let $\frac{1}{2}< p \leq 1$, $n$ a positive integer, and $w \in L^q(D,\bc)$, $q>1$. The Hilbert transform is a continuous operator on $(H^{n,p}_{w}(D,\bc))_b$ with the quasi-norm $||\cdot||_{\bc,b}$.  
\end{theorem}

\begin{proof}
By Theorem \ref{Thm: higheratomicbc}, every $f \in H^{n,p}_{w}(D,\bc)$ has a boundary value in the sense of distributions $f_b$, and there exist $\{c_n^+\}, \{c_n^-\} \in \ell^p(\C)$, $\{a_n^+\}$ and $\{a_n^-\}$ are collections of $p$-atoms such that
\[
    f_b = p^+ \left( \sum_{n=1}^\infty c^+_n a^+_n\right)^* + p^- \sum_{k=1}^\infty c^-_k a^-_k + \Psi_b
\] 
where $\Psi_b \in L^\gamma(\p D, \bc)$, $1 < \gamma< \frac{q}{2-q}$. By Theorem \ref{Thm: bcnonhomoghhilbertrans}, 
\[
    ||H(p^+ \left( \sum_{n=1}^\infty c^+_n a^+_n\right)^* + p^- \sum_{k=1}^\infty c^-_k a^-_k)||_{L^p(\p D)} \leq C_{at}||p^+ \left( \sum_{n=1}^\infty c^+_n a^+_n\right)^* + p^- \sum_{k=1}^\infty c^-_k a^-_k||_{\bc, at},
\]
where $C_{at}$ is a constant. By Theorem \ref{hilbertpgreaterthanone},
\[
    ||H(\Psi_b)||_{L^p(\p D)} \leq C_\psi ||\Psi_b||_{L^\gamma(\p D)}.
\]
Therefore, 
\begin{align*}
    ||H(f_b)||_{L^p(\p D)} 
    &\leq ||H(p^+ \left( \sum_{n=1}^\infty c^+_n a^+_n\right)^* + p^- \sum_{k=1}^\infty c^-_k a^-_k)||_{L^p(\p D)} + ||H(\Psi_b)||_{L^p(\p D)} \\
    &\leq C_{at}||p^+ \left( \sum_{n=1}^\infty c^+_n a^+_n\right)^* + p^- \sum_{k=1}^\infty c^-_k a^-_k||_{\bc, at} + C_\psi ||\Psi_b||_{L^\gamma(\p D)} \\
    &\leq C\left(||p^+ \left( \sum_{n=1}^\infty c^+_n a^+_n\right)^* + p^- \sum_{k=1}^\infty c^-_k a^-_k||_{\bc, at} +  ||\Psi_b||_{L^\gamma(\p D)} \right) \\
    &= C||f_b||_{\bc,b},
\end{align*}
where $C$ is a constant that dominates $C_{at}$ and $C_{\Psi}$. 

\end{proof}

\section*{Acknowledgments}

The author thanks Gustavo Hoepfner and the Departamento de Matemática da Universidade Federal de São Carlos for their hospitality while the author was visiting in the summer of 2023. During this visit, the author realized the extension of Theorem 2.10 from \cite{WB} that is Theorem \ref{thm: 2.10WBextension} above. Also, the author expresses their gratitude to the anonymous referees for their valuable comments and suggestions that improved the quality of this article.

\end{document}